\documentclass{amsart}
\usepackage{amsmath,amssymb,dsfont}
\usepackage{graphicx,color}
\usepackage{pstricks, pst-plot, pst-node, pst-grad, pst-math}
\usepackage{pst-plot}
\usepackage{animate}

\newtheorem{theorem}{Theorem}[section]

\newtheorem{proposition}[theorem]{Proposition}
\newtheorem{corollary}[theorem]{Corollary}

\theoremstyle{definition}
\newtheorem{definition}[theorem]{Definition}

\newtheorem{example}[theorem]{Example}
\newtheorem{remark}[theorem]{Remark}
\newtheorem{hypothesis}[theorem]{Hypothesis}

\newcommand\lm{\lambda}

\newcommand\al{\alpha}
\newcommand\be{\beta}

\newcommand\gm{\gamma}

\newcommand{\rd}{{\,\rm d}}
\newcommand{\e}{{\rm e}}

\newcommand{\N}{{\mathbb N}}

\newcommand{\R}{{\mathbb R}}
\newcommand{\C}{{\mathbb C}}

\newcommand\eps{\varepsilon}

\newcommand\beq{\begin{equation}}
\newcommand\eeq{\end{equation}}

\newcommand{\dom}{\mathcal{D}}

\newcommand{\dist}{\mathrm{dist}}
\newcommand\re{\mathrm{Re}}
\newcommand\im{\mathrm{Im}}
\newcommand\I{\mathrm{i}}
\newcommand{\beqnt}{\begin{equation*}}
\newcommand{\eeqnt}{\end{equation*}}
\newcommand{\set}[2]{\{#1 : #2 \}}
\newcommand{\Set}[2]{\Big\{  #1 : #2 \Big\}  }
\newcommand{\SSet}[2]{\left\{  #1 : #2 \right\}  }

\renewcommand\H{\mathcal{H}}
\newcommand\K{\mathcal{K}}
\newcommand\A{\mathcal{A}}
\newcommand{\slim}{\operatorname{s\,-}\lim}

\newcommand{\LL}{\operatorname{L}}
\newcommand{\sgn}{\operatorname{sgn}}
\newcommand\rref[1]{{\rm \ref{#1}}}
\DeclareMathOperator{\Tr}{Tr}
\DeclareMathOperator{\tr}{tr}

\newcommand{\dotcup}{\ensuremath{\mathaccent\cdot\cup}}

\newcommand{\mycircle}[1]{%draws two circles
\psaxes(0,0)(-5,-2.5)(5,2.5)
\psframe[fillstyle=solid,linecolor=red,fillcolor=red](-5,-.05)(-1,.05)
\psframe[fillstyle=solid,linecolor=red,fillcolor=red](5,-.05)(1,.05)
\uput[u](-4.2,0){$\textcolor{red}{\sigma_{\rm e}(H)}$}
\uput[u](4.2,0){$\textcolor{red}{\sigma_{\rm e}(H)}$}
\parametricplot[plotpoints=200,%
plotstyle=curve]%
{0}{360}
{
#1 #1 mul #1 #1 mul mul #1 #1 mul 2 mul sub 2 add 1 #1 #1 mul sub 4 mul div 0.5 add sqrt #1 #1 mul #1 #1 mul mul #1 #1 mul 2 mul sub 2 add 1 #1 #1 mul sub 4 mul div 0.5 sub sqrt t sin mul add
#1 #1 mul #1 #1 mul mul #1 #1 mul 2 mul sub 2 add 1 #1 #1 mul sub 4 mul div 0.5 sub sqrt t cos mul 
}
\parametricplot[plotpoints=200,%
plotstyle=curve]%
{0}{360}
{
#1 #1 mul #1 #1 mul mul #1 #1 mul 2 mul sub 2 add 1 #1 #1 mul sub 4 mul div 0.5 sub sqrt t sin mul #1 #1 mul #1 #1 mul mul #1 #1 mul 2 mul sub 2 add 1 #1 #1 mul sub 4 mul div 0.5 add sqrt sub
#1 #1 mul #1 #1 mul mul #1 #1 mul 2 mul sub 2 add 1 #1 #1 mul sub 4 mul div 0.5 sub sqrt t cos mul 
}
}
\begin{document}

\title[Eigenvalue estimates for non-selfadjoint Dirac operators]
{Eigenvalue estimates for non-selfadjoint Dirac operators on the real line}

\author{Jean-Claude Cuenin}
\address{Department of Mathematics, Imperial College London, London SW7 2AZ, UK}
\email{j.cuenin@imperial.ac.uk}

\author{Ari Laptev}
\address{Department of Mathematics, Imperial College London, London SW7 2AZ, UK}
\email{a.laptev@imperial.ac.uk}

\author{Christiane Tretter}
\address{Mathematisches Institut, Universit\"at Bern, Sidlerstr.\ 5, 3012 Bern, Switzerland}
\email{tretter@math.unibe.ch}

%\subjclass{ }
%\date{\today}

\begin{abstract}
We show that the non-embedded eigenvalues of the Dirac operator on the real line with non-Hermitian potential $V$
lie in the disjoint union of two disks in the right and left half plane, respectively, provided that the 
$L^1\mbox{-norm}$ of $V$ is bounded from above by the speed of light times the reduced Planck constant.
An analogous result for the Schr\"odinger operator, originally proved by Abramov, Aslanyan and Davies, emerges in the nonrelativistic limit. 
For massless Dirac operators, the condition on $V$ implies the absence of nonreal eigenvalues. Our results are further generalized to potentials with slower decay at infinity.
As an application, we determine bounds on resonances and embedded eigenvalues of Dirac operators with Hermitian dilation-analytic potentials.
\end{abstract}

\maketitle

%%%%%%%%%%%%%%%%%%%%%%%%%%%%%%%%

\section{Introduction} 

There has been an increasing interest in the spectral theory of non-selfadjoint differential operators during the past years. In particular, eigenvalue estimates for Schr\"odinger operators with complex potentials have recently been investigated by various authors, \cite{AAD01,DaNa02,FrLaLiSe06,LaSa09,Sa10,Frank10}. Corresponding results for non-selfadjoint Dirac operators are much more sparse, \cite{Syroid83, Syroid87}, although operators of this type arise for example as Lax operators in the focusing nonlinear Schr\"odinger equation \cite{RGHL04}. 

%The aim of this paper is to obtain eigenvalue bounds for Dirac operators on the real line with non-Hermitian potentials.
In this paper we derive the first eigenvalue enclosures for Dirac operators with non-Hermitian potentials. We consider one-dimensional Dirac operators $H\!=\!H_0+V$ in $L^2(\R)\otimes \C^2$, where the free Dirac operator is of the form
\begin{align}\label{eq. Dirac op.}
H_0=-\I c\hbar \,\frac{\rd}{\rd x}\,\sigma_1+m c^2\,\sigma_3,\quad 
\sigma_1:=\begin{pmatrix}0&1\\ 1&0\end{pmatrix},\quad
\sigma_3:=\begin{pmatrix}1&0\\ 0&-1\end{pmatrix}
%H=H_0+V,\quad H_0=\begin{pmatrix}mc^2&-\I c \hbar\frac{\rd}{\rd x}\\[1mm]-\I c \hbar\frac{\rd}{\rd x}&-mc^2\end{pmatrix}
\end{align}
with $c$ denoting the speed of light, $\hbar$ the reduced Planck constant, $m\geq 0$ the particle mass 
and where $V$ is a $2\times 2$ matrix-valued function with entries in $L^1(\R)$. 
Since we do not assume $V(x)$ to be Hermitian, the operator $H$ is not selfadjoint, in general.  
Moreover, already the free Dirac operator $H_0$ is not bounded from below, with
purely absolutely continuous
$
\sigma(H_0)=(-\infty,-mc^2]\cup [mc^2,\infty)
$.

In our main result, Theorem \ref{thm. 1}, we prove that if the potential $V$ satisfies 
\beq\label{eq. assumption on V trace}
\|V\|_1 := \int_{\R}\|V(x)\|\, \rd x <\hbar c,
\eeq
where $\|V(x)\|$ is the operator norm of $V(x)$ in $\C^2$ with Euclidean norm, 
%
%
%\[
%\|V\|_1:=\int_{\R}\|V(x)\|\, \rd x,\quad \|V(x)\|:=\sup_{u\in\C^2\setminus\{0\}}\frac{\|V(x)u\|_{\C^2}}{\|u\|_{\C^2}},
%\]
%
then the non-embedded eigenvalues of $H$ lie in the union of two disjoint disks, 
\beq\label{eq. spec. inclusion intro}
\sigma_{\rm d}(H)\subset K_{mr_0}(mx_0)\,\dotcup \,K_{mr_0}(-mx_0),
\eeq
in the right and left half plane; the radii $mr_0$, as well as the points $mx_0$ determining the centres, diverge to $\infty$ as $\|V\|_1\to \hbar c$. In particular, our theorem implies that the massless Dirac operator (i.e.\ $m=0$ in \eqref{eq. Dirac op.}) with non-Hermitian potential $V$ has no complex eigenvalues at all since in this case $mr_0=0$. 

The second main result of this paper is an enclosure for resonances of Dirac operators with Hermitian potentials under some analyticity assumptions on~$V$. While the literature on the theory of resonances of Schr\"odinger operators is vast, see e.g.\ \cite{SZ07}, \cite{Z02}
and the references therein, much less is known for the Dirac operator; we only mention \cite{Seba88} where the complex scaling method was employed. We use the interplay of this method with Theorem \ref{thm. 1} for the scaled Dirac operators $H_\theta$ to describe a region in the complex plane where the uncovered resonances may lie in terms of $L_1$-norms of the scaled potentials $V(\e^{\I\theta} \cdot)$. Moreover, for the massless Dirac operator, we show that there are no resonances near the real axis.

Further results concern the sharpness of our eigenvalue enclosures and generalizations to more slowly decaying potentials. 
Finally, in the non-relativistic limit ($c\to\infty$), our main result reproduces \cite[Theorem 4]{AAD01} for the one-dimensional Schr\"odinger operator
 \beq\label{eq. Schroedinger op.}
 -\frac{\hbar^2}{2m}\frac{\rd^2}{\rd x^2}+V
 \eeq
in $L^2(\R)$ with complex-valued potential $V\in L^1(\R)$
the eigenvalues $\lm\in\C\setminus [0,\infty)$ of which lie in a disk around the origin:
 \beq\label{eq. disk Schroedinger}
\frac{\hbar^2}{2m}|\lm|\leq \frac{1}{4}\left(\int_{\R}|V(x)|\rd x\right)^2.
 \eeq

Our proofs are based on the so-called Birman-Schwinger principle. 
%with the explicit form of the resolvent kernel of the free Dirac operator $H_0$.
Although the latter is not bound to one dimension, the generalization to higher dimensions poses a major challenge; the reason for this is the intrinsically different behaviour of the resolvent kernel of $H_0$ which already in the case of Schr\"odinger operators requires  sophisticated analytical estimates \cite{Frank10}.

The outline of the paper is as follows. In Theorem \ref{thm. 1} of Section \ref{section Integrable potentials}, we prove the enclosure \eqref{eq. spec. inclusion intro} and show that, for $m\neq 0$, the eigenvalue bound \eqref{eq. disk Schroedinger} for the Schr\"odinger operator emerges in the nonrelativistic limit ($c\to\infty$). One of the main, new, ingredients in the proof of Theorem \ref{thm. 1} is the use of a M\"obius transformation of the spectral parameter to localize the eigenvalues.

In Section \ref{section sharpness}, we demonstrate the sharpness of Theorem \ref{thm. 1} by considering a family of delta-potentials. 
%We also prove a special case of Theorem \ref{thm. 1} for purely imaginary potentials under somewhat weaker assumptions than 
Moreover, we show that assumption \eqref{eq. assumption on V trace} may be weakened if the potential has additional structure such as being purely imaginary.

In Section \ref{section Slowly decaying potentials}, we extend Theorem \ref{thm. 1} to potentials with slower decay at infinity; in this case  \eqref{eq. assumption on V trace} has to be replaced by more complicated conditions. From this we derive eigenvalue estimates in terms of higher $L^p$-norms of $V$, see Corollary \ref{cor. higher Lp}. We also prove that, if $p\in[2,\infty]$ and an additional smallness assumption holds, then $H$ is similar to a block-diagonal matrix operator, see Theorem \ref{thm. blockdiag}.

In Section \ref{section Embedded eigenvalues and resonances}, we establish enclosures for resonances and embedded eigenvalues of $H$ with Hermitian $V(x)$. For this purpose, we use the well-known method of complex scaling where resonances are characterized as eigenvalues of non-selfadjoint operators and apply Theorem \ref{thm. 1} to the scaled Dirac operators $H_\theta$. To this end, a careful analysis of the dependence of the corresponding balls $K_{mr_\theta}(\pm mx_\theta)$ on the scaling angle $\theta$ is required.  

To avoid overly technical discussions, we prove all results in Sections \ref{section Integrable potentials}--5 for the case 
of bounded $V$, i.e.\ $V_{ij}\in L^{\infty}(\R)$, $i,j=1,2$; 
%This restriction is imposed for the sole purpose of avoiding technicalities in the proofs;
it will be evident, however, that the boundedness does not play an essential role, and we will show in Section \ref{section Construction of the perturbed operator} how to dispense with it. 

The following notation will be used throughout this paper. For $z_0\in\C$ and $r>0$, let $K_{r}(z_0)$ be the closed disk centred at $z_0$ with radius~$r$; for $r=0$, we use the convention that $K_r(z_0)=\emptyset$. For a closed densely defined linear operator $T:\H\to\H$ on a Hilbert space $\H$, we denote by $\dom(T)$, $\ker(T)$, $\rho(T)$, $\sigma(T)$, $\sigma_{\rm p}(T)$ its domain, kernel, resolvent set, spectrum, and set of eigenvalues, respectively. 
% The numerical range of $T$ is defined as
%
% \[
% W(H):=\set{(Tf,f)}{f\in\dom(T),\,\|f\|=1},
% \] 
%
% where $(\cdot,\cdot)$ denotes the scalar product in $\H$.
Let $\LL(\H)$ denote the algebra of bounded linear operators with domain equal to $\H$ and by $\|\cdot\|$ the operator norm on $\LL(\H)$; the norm on the ideal of Hilbert-Schmidt operators is denoted by $\|\cdot\|_{\rm HS}$. The identity operator on $\H$ is denoted by $I_{\H}$. We shall use the abbreviation $T-z$ for the operator $T-z\,I_{\H}$, $z\in\C$. Throughout Sections \ref{section Integrable potentials}--\ref{section Embedded eigenvalues and resonances} we work in the Hilbert space $\H=L^2(\R)\otimes \C^2$. By $\tr$ we denote the trace in this Hilbert space, while $\Tr$ is the trace in $\C^2$. By abuse of notation, we shall denote integral operators on $\H$ and their kernels by the same symbol. For example, we write $R_0(z)=(H_0-z)^{-1}$ for the resolvent of the free Dirac operator $H_0$ and $R_0(x,y;z)$ for its resolvent kernel. 
For a measurable matrix-valued function $V=(V_{ij})_{i,j=1}^2$ 
%with $V_{ij}\in L^1(\R)$ 
%we set
%\[
%\|V\|_1:=\int_{\R}\|V(x)\|\,\rd x, \quad 
%\|V(x)\|:=\sup_{u\in\C^2\setminus\{0\}}\frac{\|V(x)u\|_{\C^2}}{\|u\|_{\C^2}},
%\]
%where $\|\cdot\|_{\C^2}$ is the Euclidean norm on $\C^2$. 
we shall always identify the function $V$ with the closed maximal multiplication operator in $L^2(\R)\otimes \C^2$.

The potentials $V$ we consider are supposed to decay at infinity, 
\[
\lim_{|x|\to\infty}V_{ij}(x)\to 0,\quad |x|\to \infty.
\]
It is well known that the essential spectrum of $H_0$ is stable under such perturbations,
\beq\label{eq. invariance essential spectrum}
\sigma_{\rm e}(H)=\sigma_{\rm e}(H_0)=(-\infty,-mc^2]\cup [mc^2,\infty), 
\eeq
see e.g.\ \cite[4.3.4, Remark 2]{Th}. Note that there are at least five different notions of essential spectrum for a non-selfadjoint closed operator $T$; here we use the following one:
\begin{align*}
\sigma_{\rm e}(T):=\set{z\in\C}{T-z \mbox{ is not a Fredholm operator}}.
\end{align*}
With this definition of the essential spectrum, it follows from \cite[Theorem IX.2.4]{EE} that \cite[4.3.4, Remark 2]{Th}, which is only stated for Hermitian-valued potentials, still holds true in the non-Hermitian case.
The discrete spectrum of $T$ is defined as
\[
\sigma_{\rm d}(T):=\set{z\in\C}{z\mbox{ is an isolated eigenvalue of $T$ of finite multiplicity}}.
\]
Note that, if $T$ is not selfadjoint, then, in general, $\sigma(T)$ is not the disjoint union of $\sigma_{\rm e}(T)$ and $\sigma_{\rm d}(T)$. However, for the Dirac operators $H=H_0+V$ considered here, $\C\setminus\sigma_{\rm e}(H_0)=\rho(H_0)$ has either one or two (for $m=0$) connected components, each of which contains points of $\rho(H)$. Hence \cite[Theorem XVII.2.1]{GGK1} implies that 
\beq\label{eq. complement of essential spectrum is discrete}
\sigma(H)\setminus\sigma_{\rm e}(H)=\sigma_{\rm d}(H).
\eeq
%
%We shall also denote the set of all eigenvalues of $T$ by $\sigma_{\rm p}(T)$. 

For simplicity, we will use units where $\hbar=c=1$ from now on. The correct values in other units may simply be restored by dimensional analysis.

\section{Integrable potentials}\label{section Integrable potentials}

In this section we derive sharp bounds on the eigenvalues of the perturbed Dirac operator $H$ in \eqref{eq. Dirac op.}, with potential $V=(V_{ij})_{i,j=1}^2$, $V_{ij}\in L^1(\R)$.
%  As already mentioned in the introduction, we also assume for simplicity that $V_{ij}\in L^{\infty}(\R)$, which implies that $H$ is a closed operator on $\dom(H_0)$. This purely technical assumption will be removed in Section \ref{section Construction of the perturbed operator}, which is why we are only interested in $L^1$-bounds in this section (as opposed to $L^{\infty}$-bounds). 
For eigenvalue bounds in terms of higher $L^p$-norms see Corollary \ref{cor. higher Lp} as well as the forthcoming paper  \cite{CueTre12}. 

\begin{theorem}\label{thm. 1}
Let $V=(V_{ij})_{i,j=1}^2$ with $V_{ij}\in L^1(\R)$ for $i,j=1,2$ be such that
\beq\label{eq. v less than c}
\|V\|_1 <1.
\eeq
Then
\beq\label{eq. spectrum in two disks}
\sigma_{\rm d}(H)\subset K_{mr_0}(mx_0)\,\dotcup \,K_{mr_0}(-mx_0),
\eeq
where
\beq\label{eq. x0 r0}
x_0:=\sqrt{\frac{\|V\|_1^4-2 \|V\|_1^2+2}{4(1-\|V\|_1^2)}+\frac{1}{2}}, \quad r_0:=\sqrt{\frac{\|V\|_1^4-2 \|V\|_1^2+2}{4(1-\|V\|_1^2)}-\frac{1}{2}};
\eeq
in particular, the spectrum of the massless Dirac operator {\rm(}$m=0${\rm)} with non-Hermitian potential $V$ is $\R$. 
\end{theorem}

%%%%%%%%%%%%%%%%%% Animation (for online version)
% \begin{figure}[ht]
% \begin{animateinline}[poster=last, controls]{8}%
% \multiframe{9}{Rv=0+0.1}
% {
% \begin{pspicture}(-6,-2.5)(6,2.5)
% \psaxes(0,0)(-5,-2.5)(5,2.5)
% \mycircle{\Rv}
% \end{pspicture}
% }
% \newframe
% \multiframe{3}{Rv=0.83+0.03}
% {
% \begin{pspicture}(-6,-2.5)(6,2.5)
% \psaxes(0,0)(-5,-2.5)(5,2.5)
% \mycircle{\Rv}
% \end{pspicture}
% }
% \newframe
% \multiframe{10}{Rv=0.9+0.01}
% {
% \begin{pspicture}(-6,-2.5)(6,2.5)
% \psaxes(0,0)(-5,-2.5)(5,2.5)
% \mycircle{\Rv}
% \end{pspicture}
% }
% \end{animateinline}
% \caption{The two disks of Theorem \ref{thm. 1} for different values of $\|V\|_1\in(0,1)$ and $m=1$; the imaginary axis always remains free of eigenvalues.}
% \end{figure}
%%%%%%%%%%%%%%%%%%

%%%%%%%%%%%%%%%%% Figure (for print version)
\begin{figure}[ht]
\begin{pspicture}(-6,-2.5)(6,2.5)
\psaxes(0,0)(-5,-2.5)(5,2.5)
\mycircle{0.8}
\mycircle{0.96}
\mycircle{0.978}
\end{pspicture}
\caption{The two disks of Theorem \ref{thm. 1} for three different values of $\|V\|_1\in(0,1)$ and $m=1$; the imaginary axis always remains free of eigenvalues.}
\end{figure}
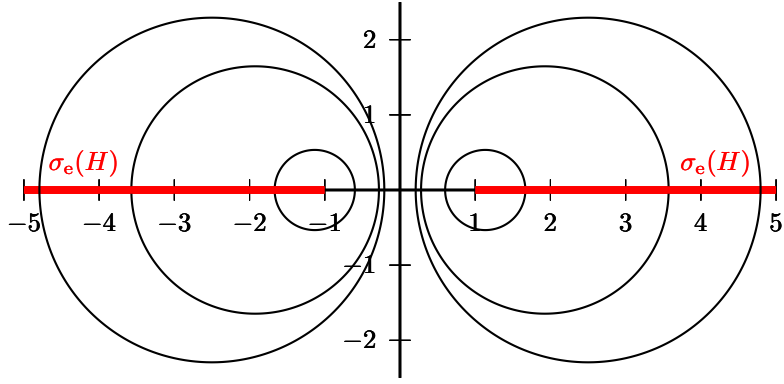
%%%%%%%%%%%%%%%%%

\begin{proof}%[Proof of Theorem \rref{thm. 1}]
In this section we prove Theorem \rref{thm. 1} under the assumption that $V$ is bounded in which case $H=H_0+V$ is a closed operator. The only additional obstruction in the general case is the construction of a closed extension $H$ of $H_0+V$, a technical point which we postpone to Section \ref{section Construction of the perturbed operator}.

The proof of Theorem \rref{thm. 1} is based on the Birman-Schwinger principle: Let $U$ be the partial isometry in the polar decomposition of $V=U|V|$. We shall factorize $V$ according to 
\beq\label{eq. standard factorization}
V=BA,\quad B:=U|V|^{1/2},\quad A:=|V|^{1/2}.
\eeq
We denote by $R_0(\cdot)$ the resolvent of $H_0$, i.e.\
\begin{align*}
R_0(z)&:=(H_0-z)^{-1},\quad z\in\rho(H_0).
\end{align*}
Let $z\in\rho(H_0)$. It is easy to verify that $z$ is an eigenvalue of $H$ if and only if $-1$ is an eigenvalue of $VR_0(r)$. Since the nonzero eigenvalues of $BAR_0(z)$ and $AR_0(z)B$ are the same, this is thus equivalent to $-1$ being an eigenvalue of the operator
\beq\label{eq. def. Q}
Q(z):=AR_0(z)B:\H\to\H,\quad z\in\rho(H_0).
\eeq
Hence, if $z$ is an eigenvalue of $H$, then $\|Q(z)\|\geq 1$. On the other hand, since the spectrum of $H$ in the complement of $\sigma_{\rm e}(H_0)$ is discrete by \eqref{eq. invariance essential spectrum} and \eqref{eq. complement of essential spectrum is discrete}, $z\in\rho(H)$ whenever $\|Q(z)\|<1$. 
%Using the general version of the Birman-Schwinger principle as stated in Theorem \ref{abstract theorem}, the proof given here can easily be adapted to unbounded potentials, see Section~\ref{section Construction of the perturbed operator}

It is well-known that the resolvent kernel of the free Dirac operator is given by
\begin{align*}
 R_0(x,y;z)
=M(x,y;z)\,\e^{\I k(z)|x-y|},\quad M(x,y;z):=\frac{\I}{2}\begin{pmatrix}\zeta(z)&\sgn(x-y)\\ \sgn(x-y) &\zeta(z)^{-1}\end{pmatrix},
\end{align*}
where
\beq\label{zeta, k}
 \zeta(z):=\frac{z+m}{k(z)},\quad  k(z):=\sqrt{z^2-m^2},\quad z\in \rho(H_0),
\eeq
and the branch of the square root on $\C\setminus [0,\infty)$ is chosen such that $\im\, k(z)>0$. We set 
\beq\label{Phi, eta}
\begin{split}
\Phi(z)&:=\zeta(z)^2=\frac{z+m}{z-m}\in\C\setminus[0,\infty),\quad z\in\rho(H_0),\\
\eta(s)&:=\sqrt{\frac{1}{2}+\frac{1}{4}\left(s+s^{-1}\right)},\quad s>0.
\end{split}
\eeq
Observing that
\[
\|M(x,y;z)\|=\|M(x,y;z)\|_{\rm HS}=\eta(|\Phi(z)|),
\]
we obtain that for $z\in \rho(H_0)$, $f,g\in \H$,
\begin{align}
|(AR_0(z)Bf,g)|&\leq \eta(|\Phi(z)|)\, \int_{\R}\int_{\R} \|A(x)\| \, \|B(y)\| \, \|f(y)\|_{\C^2} \|g(x)\|_{\C^2}\rd x\rd y\notag\\
&\leq \eta(|\Phi(z)|) \left(\int_{\R} \|A(x)\|^2 \rd x\right)^{1/2}\|g\|_{\H} \left(\int_{\R} \|B(y)\|^2 \rd y\right)^{1/2}\|f\|_{\H}\label{eq. Cauchy-Schwarz AB}\\
&= \eta(|\Phi(z)|)\, \left(\int_{\R} \|V(x)\| \rd x\right) \, \|g\|_{\H} \,\|f\|_{\H}\notag.
\end{align}
Here, we used $\exp(-\im\,k(z)\,|x-y|)\leq 1$ in the first line, the Cauchy-Schwarz inequality in the second line, and the equality
\begin{align*}
&\|B(x)\|=\|A(x)\|=\|\,|V(x)|^{1/2}\|=\|V(x)\|^{1/2},\quad x\in\R,
\end{align*}
in the last line. It follows that
\begin{align}\label{eq. Qleqmv}
\|Q(z)\|\leq \eta(|\Phi(z)|)\, \|V\|_1. 
\end{align}
Hence, $\|Q(z)\|<1$ whenever
\beq\label{eq. w and rho}
w:=\Phi(z)\in B_{\rho^2,\rho^{-2}}:=\set{w\in\C}{\rho^{-2}<|w|<\rho^2},\quad \rho:=\frac{1+\sqrt{1-\|V\|_1^2}}{\|V\|_1}.
\eeq
Observing that $\Phi$ is a M\"obius transformation for $m\neq 0$ with inverse
\[
z=\Phi^{-1}(w)=m\,\frac{w+1}{w-1}, 
\]
we see that the complement of the annulus $B_{\rho^2,\rho^{-2}}$ in the $w$-plane is mapped onto the union of the disks $K_{mr_0}(mx_0)$ and $K_{mr_0}(-mx_0)$ in the $z$-plane.
Indeed, $\Phi^{-1}$ maps (generalized) circles to (generalized) circles, and, by virtue of the equality
\[
\overline{\Phi^{-1}(w)}=\Phi^{-1}(\overline{w}),\quad w\in\C\cup\{\infty\},
\]
the image of a circle with centre at the origin is symmetric with respect to the real axis. The outer boundary of $B_{\rho^2,\rho^{-2}}$ is mapped to the circle with centre $mx_0$ and radius $mr_0$ given by
\[
x_0=\frac{1}{2}\left(\frac{\rho^2+1}{\rho^2-1}+\frac{-\rho^2+1}{-\rho^2-1}\right)=\frac{\rho^4+1}{\rho^4-1},\quad r_0=\frac{1}{2}\left(\frac{\rho^2+1}{\rho^2-1}-\frac{-\rho^2+1}{-\rho^2-1}\right)=\sqrt{x_0^2-1}.
\]
On the other hand, since
\[
\Phi^{-1}(w^{-1})=-\Phi^{-1}(w),\quad w\in\C\cup\{\infty\},
\]
the inner boundary of $B_{\rho^2,\rho^{-2}}$ is mapped to the circle with centre $-mx_0$ and radius~$mr_0$.
Since $\Phi^{-1}$ is biholomorphic and $\C\setminus\left(B_{\rho^2,\rho^{-2}}\right)$ is doubly connected, its image must be too, so it fills the regions inside the two circles. Observing that 
\[
\frac{\rho^4+1}{\rho^4-1}=\sqrt{\frac{\|V\|_1^4-2 \|V\|_1^2+2}{4(1-\|V\|_1^2)}+\frac{1}{2}},
\]
the spectral inclusion \eqref{eq. spectrum in two disks} is proved for the case $m\neq 0$. If $m=0$, then $\Phi(z)=1$ and $\eta(|\Phi(z)|)=1/c$ for all $z\in\C$. Hence, \eqref{eq. Qleqmv}  implies that $\|Q(z)\|<1$ for $z\in\rho(H_0)=\C\setminus\R$. This proves the limiting case $m=0$ in \eqref{eq. spectrum in two disks}.  
\end{proof}

\begin{remark}
The eigenvalue bound \eqref{eq. disk Schroedinger} of \cite{AAD01} for the Schr\"odinger operator with complex potential $V$ emerges from the corresponding bounds for the Dirac operator \eqref{eq. spectrum in two disks} in the nonrelativistic limit since
\[
\lim_{c\to\infty}(H(c)-mc^2-z)^{-1}=\begin{pmatrix}\left(-\frac{1}{2m}\,\Delta+V-z\right)^{-1}&0\\0&0\end{pmatrix},
\]
see e.g.\ \cite[Theorem 6.4]{Th}. Here, we have restored $c$ (the speed of light) by replacing $m$ by $mc^2$ and $c^{-1}\|V\|_1$ and $\|V\|_1$. It follows from Theorem \ref{thm. 1} that
\beq\label{eq. c-dependence}
\sigma_{\rm d}(H(c)-mc^2)\subset K_{mc^2r_0(c)}\big(mc^2(x_0(c)-1)\big)\,\dotcup \,K_{mc^2r_0(c)}\big(mc^2(x_0(c)+1)\big),
\eeq
where $x_0(c)$, $r_0(c)$ now depend on $c$ via $c^{-1}\|V\|_1$. An easy calculation shows that, in the limit $c\to\infty$, the right hand side of \eqref{eq. c-dependence} converges to the closed disk with radius $m/2\,\|V\|_1^2$ and centre at the origin, compare \eqref{eq. disk Schroedinger} (recall that $\hbar=1$ here). 
\end{remark}

\begin{remark}
The factorization of $V$ used above is optimal in the sense that for an arbitrary factorization $V=B'A'$, the last equality in \eqref{eq. Cauchy-Schwarz AB} generally turns into an inequality.
We also note that
\[
\|V\|_1\leq \int_{\R} \|V(x)\|_{\rm HS} \rd x =\int_{\R}\sqrt{\Tr V(x)^*V(x)}\rd x.
\]
\end{remark}

\begin{remark} 
For the massless Dirac operator ($m=0$) the absence of nonreal eigenvalues can also be proved by showing that the perturbed operator $H$ is similar to the selfadjoint operator $H_0$:

i) For potentials of the particular form 
\beq\label{eq. Syroid}
V=\begin{pmatrix}v_1&v_2\\v_2&-v_1\end{pmatrix}
\eeq
with $v_1,v_2\in L^1(\R)$, Syroid proved in \cite{Syroid83} by a method due to Kato \cite{Ka65} that, if $\|V\|_1<1$, then $H$ is similar to $H_0$,
\[
H=W_{\pm}H_0 W_{\pm}^{-1}.
\]
Here, $W_{\pm}$ are the Kato wave operators \cite{Ka65}, which admit the representation
\[
W_{\pm}=\slim_{t\to\pm\infty}\e^{\I t H}\e^{-\I t H_0},
\]
and $\e^{\I t H}$, $\e^{\I t H_0}$ are the strongly continuous groups with generators $H$, $H_0$, respectively. In particular, $H$ is a spectral operator in the sense of \cite{DS3} with absolutely continuous spectrum $\sigma(H)=\R$.
It is not difficult to check that the proof in \cite{Syroid83} also works without the assumption~\eqref{eq. Syroid}. 
%This, together with \cite{Ka65} yields another proof for the absence of complex eigenvalues in the case $m=0$.

ii) If $V$ is an electric potential of the form
\[
V=\begin{pmatrix}q&0\\0&q\end{pmatrix}
\]
with a complex-valued function $q\in L^1(\R)$, then the similarity of $H$ and $H_0$ (with $m=0$) holds without the assumption $\|V\|_1<1$. Indeed, if $U$ is the operator of multiplication with 
\[
U(x)=\exp\left(\I\sigma_1\int_{-\infty}^x q(y)\rd y\right),\quad x\in\R, %\quad \sigma_1:=\begin{pmatrix}0&1\\ 1&0\end{pmatrix}.
\]
then $U$ is bounded and boundedly invertible in $\H$, and $U^{-1}H_0U=H$. Moreover, for $z\in\C$ with $\im\,z\neq 0$, the resolvent $R(z):=(H-z)^{-1}$ can be estimated by:
\beq\label{eq. resolvent norm scalar}
\|R(z)\|\leq \|U\|\,\|U^{-1}\|\,\|R_0(z)\|\leq |\im\, z|^{-1}\exp\left(2\int_{\R} |q(y)|\rd y\right).
\eeq
\end{remark}

In analogy to \eqref{eq. resolvent norm scalar}, the following proposition provides an estimate for the norm of the resolvent $R(z)$ of $H$ for general potentials $V$.

\begin{proposition}
Let $V=(V_{ij})_{i,j=1}^2$ with $V_{ij}\in L^1(\R)$ for $i,j=1,2$ be such that $\|V\|_1 <1$. Then, 
for $z\in\rho(H_0)=\C\setminus\big((-\infty,-m]\cup [m,\infty)\big)$ outside the union of the two disks $K_{mr_0}(mx_0)$ and $K_{mr_0}(-mx_0)$,
\begin{align}\label{eq. resolvent norm general}
\|R(z)\|\leq \frac{1}{\dist(z,\sigma(H_0))}+\frac{\eta(|\Phi(z)|)^2}{\im\,k(z)}\frac{\|V\|_1}{1-\eta(|\Phi(z)|)\|V\|_1}.
\end{align}
\end{proposition}

\begin{proof}
By iterating the second resolvent identity,
\[
R(z)=R_0(z)-R_0(z)VR(z),\quad 
\]
we infer that
\beq\label{eq. resolvent formula bounded V}
R(z)=R_0(z)-R_0(z)B(I_{\H}+Q(z))^{-1}AR_0(z).
\eeq
It is easy to see that
\begin{align}\label{eq. norm of AR0 and R0B}
\max\{\|AR_0(z)\|_{\rm HS},\|R_0(z)B\|_{\rm HS}\}\leq \frac{\eta(|\Phi(z)|)}{\sqrt{2\im\,k(z)}} \|V\|_1^{1/2}. 
\end{align}
From \eqref{eq. resolvent formula bounded V}, the selfadjointness of $H_0$ and the Neumann series, it follows that
\[
\|R(z)\|\leq\|R_0(z)\|+\|R(z)-R_0(z)\|\leq \frac{1}{\dist(z,\sigma(H_0))}+\frac{\|AR_0(z)\|\,\|R_0(z)B\|}{1-\|Q(z)\|}.
\]
If we combine this with \eqref{eq. norm of AR0 and R0B} and \eqref{eq. Qleqmv}, the claim is proved.
\end{proof}

\section{Sharpness of Theorem \rref{thm. 1} and purely imaginary potentials}\label{section sharpness}

In this section we provide an example which suggests that the eigenvalue enclosures of Theorem \rref{thm. 1} are sharp and that the assumption $\|V\|_1<1$ cannot be omitted. Moreover, we show how additional structure of the potential may be used to improve the bounds of Theorem \rref{thm. 1}.

\begin{example}\label{example}
We consider the family of delta-potentials
\beq\label{eq. delta-potentials}
V_{\tau}=\I\,\kappa\,\delta_0 \,W_{\tau},\quad W_{\tau}:=\begin{pmatrix}\e^{\I\,\tau}&0\\0&\e^{-\I\,\tau}\end{pmatrix},\quad \kappa >0,\quad -\pi\leq\tau<\pi,
\eeq
for which the operator $Q(z)$ reduces to the matrix
\beq\label{Q delta-potential}
Q(z)=-\frac{\kappa }{2}\begin{pmatrix}\e^{\I\,\tau}\,\zeta(z)&\e^{-\I\,\tau}\\ \e^{\I\,\tau} &\e^{-\I\,\tau}\,\zeta(z)^{-1}\end{pmatrix}
\eeq
in $\C^2$ if we define $\sgn(0)=1$. The perturbed operator $H_{\tau}$ may be rigorously defined as a rank two perturbation of $H_0$. Alternatively, it may be described in terms of boundary conditions, v.i.z.\

\begin{align*}
&\dom(H_{\tau})=\set{f\in L^2(\R,\C^2)\cap H^1(\R\setminus\{0\},\C^2)}{\sigma_1(f(0+)\!-\!f(0-))-\kappa \,W_{\tau} f(0+)=0},\\
&(H_{\tau}f)(x)=-\I  \,\frac{\rd}{\rd x}\,\sigma_1\,f(x)+m\,\sigma_3\,f(x),\quad x\in\R\setminus\{0\} ,\quad f\in\dom(H_{\tau}).
\end{align*}
It follows that
\[
\ker(H_{\tau}-z)\subset \left\{\begin{pmatrix}\zeta(z)\\ \sgn(\cdot) \end{pmatrix} \e^{\I\,k(z)\,|\cdot|},\begin{pmatrix}\sgn(\cdot)\\ \zeta(z)^{-1}\end{pmatrix} \e^{\I\,k(z)\,|\cdot|}\right\},
\]
and the boundary conditions imply that $\ker(H_{\tau}-z)$ is nontrivial if and only if 
\[
\det(I+Q(z))=\det \begin{pmatrix}1-\kappa /2\,\e^{\I\,\tau}\,\zeta(z)&-\kappa /2\,\e^{-\I\,\tau}\\ -\kappa /2\,\e^{\I\,\tau} &1-\kappa /2\, \e^{-\I\,\tau}\,\zeta(z)^{-1}\end{pmatrix}=0.
\]
Solving this equation for $\zeta(z)$, we find the solutions
\beq\label{eq. zeta delta potentials}
\zeta(z)=\zeta_{\pm}:=\e^{-\I\,\tau}\,\frac{1\pm \sqrt{1-\kappa ^2}}{\kappa}.
\eeq
Recalling \eqref{zeta, k}, \eqref{Phi, eta}, it is seen that we must have $\im\,\zeta(z)<0$ for $z$ to be an eigenvalue of $H_{\tau}$.

If $\kappa<1$, then $\im\,\zeta_{\pm}<0$ if and only if $0<\tau<\pi$; in this case, as $\tau$ varies from $0$ to $\pi$, the points $w_{\pm}:=\zeta_{\pm}^2$ trace out the boundary of the annulus $B_{\rho^2,\rho^{-2}}$ with 
\[
\rho:=\frac{1+\sqrt{1-\kappa ^2}}{\kappa},
\]
which is precisely $\rho$ in \eqref{eq. w and rho} with $\|V\|_1$ replaced by $\kappa$ $(<1)$. This implies that the two eigenvalues of $H_{\tau}$, $0<\tau<\pi$, lie on the boundaries of the disks $K_{mr_0}(\pm mx_0)$ of Theorem~\ref{thm. 1}. In the case $-\pi\leq\tau\leq 0$, there are no eigenvalues.

% As shown in the proof of Theorem~\ref{thm. 1}, for $m\neq 0$, the outer circle is mapped onto the boundary of $K_{mr_0}(mx_0)$ and the inner circle is mapped onto the boundary of $K_{mr_0}(-mx_0)$ under the transformation $w\mapsto \Phi^{-1}(w)$. 

If $\kappa\geq 1$, then the square root in \eqref{eq. zeta delta potentials} becomes imaginary, and it is easily verified that $\zeta_{\pm}$ lie on the unit circle, with 
\[
\im\,\zeta_{\pm}=\frac{1}{\kappa}\left(-\sin(\tau)\pm\cos(\tau)\sqrt{\kappa^2-1}\right). %<0 \Longleftrightarrow \tan(\tau)>\pm \sqrt{\kappa^2-1}.
\]
Hence, for $m\neq 0$, there are either zero, one, or two eigenvalues; as theta varies, they cover the imaginary axis.

%, with $\arg(\zeta_{\pm})=-\tau\pm\arccos(1/\kappa)$. 
%
% \[
% z_{\pm}=m\frac{\zeta_{\pm}^2+1}{\zeta_{\pm}^2-1}
% \]
%
%
% \begin{alignat*}{2}
% z_+&=\infty \Longleftrightarrow \zeta_+=-1 & \Longleftrightarrow   \tau&=\arccos(1/\kappa),\\
% z_-&=\infty \Longleftrightarrow \zeta_-=1  & \Longleftrightarrow   \tau&=\pi-\arccos(1/\kappa).
% \end{alignat*}
%
%there are two complex conjugate eigenvalues $z_{\pm}$ situated on the imaginary axis,
%
%\[
%z_{\pm}=\pm\frac{i}{\sqrt{\kappa^2-1}},
%\]
%
%while for
%For $m=0$ we have $\zeta(z)=1$ for $\im\,z>0$ and $\zeta(z)=-1$ for $\im\,z<0$. Hence, for $\kappa<1$, we find that $\det(I+(Q(z)))\neq 0$ for all $\tau\in[0,\pi]$, which means that $H$ has no eigenvalues. However, 
A straightforward calculations shows that
\begin{align*}
\zeta_+=1 & \Longleftrightarrow   \tau=\arccos(1/\kappa),\\
\zeta_-=-1  & \Longleftrightarrow   \tau=\pi-\arccos(1/\kappa).
\end{align*}
Hence, for $m=0$, 
\[
\sigma(H_{\tau})\cap(\C\setminus\R)=\begin{cases}
                              \set{z\in\C}{\im\,z >0},\quad &\tau=\arccos(1/\kappa),\\
                              \set{z\in\C}{\im\,z <0},\quad &\tau=\pi-\arccos(1/\kappa),\\
                              \qquad\quad\emptyset,\quad &\mbox{otherwise}.
                              \end{cases}
\]
 Hence, for $\kappa\geq 1$, the eigenvalues of $H_{\tau}$ need not lie in a bounded set, and hence an enclosure as in Theorem \ref{thm. 1} cannot hold. 
%
%shows that the norm constraint \eqref{eq. v less than c} on the potential in Theorem \ref{thm. 1}, guaranteeing that the spectrum of the massless Dirac operator is real, cannot be weakened in general.
\end{example}

Incidentally, this example (with $m=0$) illustrates two typical non-selfadjoint phenomena: First, since $H_{\tau}$ is a rank two resolvent perturbation of $H_0$, the essential spectra are clearly the same, $\sigma_{\rm e}(H_{\tau})=\sigma_{\rm e}(H_0)=\R$. However, for $\tau=
\arccos(1/\kappa)$ and $\tau=\pi-\arccos(1/\kappa)$, the spectrum in $\C\setminus\R$ is not discrete, but consists of dense point spectrum in the upper or lower half-plane; this is not a contradiction to~\cite[Theorem~3.1]{GGK1} since $\C\setminus\R$ is not connected.
Secondly, although it can be shown that the mapping $\tau\mapsto H_{\tau}$ is continuous in the norm resolvent topology, for $m=0$
the spectrum $\sigma(H_{\tau})$ is lower-semidiscontinuous as a function of $\tau$ at the points $\tau=
\arccos(1/\kappa)$ and $\tau=\pi-\arccos(1/\kappa)$, compare e.g.\ \cite[IV.3.2]{Ka}.

\medskip
If the potential has additional structure, the assumption $\|V\|_1<1$ may be weakened in some cases. As an example, we consider perturbations by purely imaginary potentials $V=\I\, \widetilde{V}$ with $\widetilde{V}\geq 0$. Such potentials have been studied in~\cite{LaSa09} in the framework of Schr\"odinger operators.

\begin{theorem}\label{thm. imaginary potential}
Let $V=\I\, \widetilde{V}$, with $\widetilde{V}=(\widetilde{V}_{ij})_{i,j=1}^2$ such that $\widetilde{V}\geq 0$ and $\widetilde{V}_{ij}\in L^1(\R)$ for $i,j=1,2$. Then $\sigma_{\rm d}(H)$ lies in the open upper half plane; if $z\in\rho(H_0)=\C\setminus\big((-\infty,-m]\cup [m,\infty)\big)$ and
\beq\label{eq. imaginary potentials}
\left(\re\,\frac{z+m}{\sqrt{z^2-m^2}}\right)\|\widetilde{V}_{11}\|_1+\biggl(\re\,\frac{\sqrt{z^2-m^2}}{z+m}\biggr)\|\widetilde{V}_{22}\|_1< 2,
\eeq
%
%where $\zeta(\cdot)$ is given by \eqref{zeta, k}.
then $z\notin \sigma(H)$.
In particular, if $m=0$ and 
\beq\label{eq. sum less than 2}
\|\widetilde{V}_{11}\|_1+\|\widetilde{V}_{22}\|_1< 2,
\eeq
then the spectrum of $H$ is $\R$.
\end{theorem}

\begin{remark}
The set of points satisfying \eqref{eq. imaginary potentials} does not have such a simple form as the disks in Theorem \ref{thm. 1}. However, \eqref{eq. imaginary potentials} implies e.g.\ that for $m>0$
\[
\sigma(H)\cap\I\,\R\subset\SSet{\I\,\mu}{\mu>0,\,\frac{\sqrt{\mu^2+m^2}}{\mu}\geq \frac{\|\widetilde{V}_{11}\|_1+\|\widetilde{V}_{22}\|_1}{2}}.
\]
\end{remark}

\begin{proof}
We follow the lines of the proof of \cite[Theorem 9]{LaSa09}.
Like in the proof of Theorem \ref{thm. 1} we assume that $V$ is bounded; for the proof of the general case, see Section~\ref{section Construction of the perturbed operator}.

Let $z\in\rho(H_0)$ and $Q(z)$ be given by \eqref{eq. def. Q}, i.e.\
\[
Q(z)=\I\,\widetilde{V}^{1/2}R_0(z)\widetilde{V}^{1/2}.
\]
Using the first resolvent identity, we find
\beq\label{eq. ReQ}
\re\,Q(z)=-(\im\,z) (R_0(z)\widetilde{V}^{1/2})^*(R_0(z)\widetilde{V}^{1/2}).
\eeq
If $\im\, z\leq 0$, this implies that $\re\,Q(z)\geq 0$. Hence the numerical range 
\[
W(I+Q(z)):=\set{((I+Q(z))f,f)}{f\in\H,\,\|f\|=1},
\]
satisfies $$W(I+Q(z))\subset\set{\lm\in\C}{\re\,\lm\geq 1}.$$ Since the spectrum of a bounded operator is contained in the closure of its numerical range, see \cite[Corollary V.3.3]{Ka}, it follows that $0\in\rho(I+Q(z))$, i.e.\ $z\in\rho(H)$ for $\im\, z\leq 0$. 

To prove the second claim, assume to the contrary that $z\in\rho(H_0)$ with $\im\, z> 0$ satisfies condition \eqref{eq. imaginary potentials}, and $z\in\sigma(H)$. Then \eqref{eq. ReQ} implies that $\re\,Q(z)\leq 0$, i.e.\ the spectrum of $Q(z)$ lies in the left half plane, and $-1$ is an eigenvalue of $Q(z)$. Hence the eigenvalues $\lm_j(Q(z))$ of $Q(z)$ satisfy
\beqnt %\label{eq. sum less than -1}
\sum_{j=1}^{\infty} \re\,\lm_j(Q(z))\leq -1.
\eeqnt
%
% Using the inequality
%
% \beq\label{eq. relm lmre}
% -\sum_{j=1}^{\infty} \re\,\lm_j(Q(z))\leq \tr(\re\,Q(z)),
% \eeq
%
%see e.g.\ \cite[Corollary 1]{LaSa09} or \cite[Theorem 1]{BO08}, we then infer that
%
It follows that
\beq\label{eq. imaginary contradiction}
1\leq-\sum_{j=1}^{\infty} \re\,\lm_j(Q(z))\leq -\tr(\re\,Q(z))=-\int_{\R}\Tr(\re\,Q)(x,x;z)\,\rd x,
\eeq
where $(\re\,Q)(\cdot,\cdot;z)$ is the kernel of the operator $\re\,Q(z)$; for the proof of the second inequality we refer to \cite[Corollary 1]{LaSa09} or \cite[Theorem 1]{BO08}, see also \cite[Lemma 1]{FrLaLiSe06} for a different idea of the proof. 
Since 
\[
\re\,Q(z)=-\widetilde{V}^{1/2}\im\, R_0(z)\widetilde{V}^{1/2},
\]
we have
\[
(\re\,Q)(x,x;z)=-\frac{1}{2}\widetilde{V}(x)^{1/2}\begin{pmatrix}\re\,\zeta(z)&0\\0&\re\,\zeta(z)^{-1}\end{pmatrix}\widetilde{V}(x)^{1/2}.
\]
Together with assumption \eqref{eq. imaginary potentials}, this implies
\[
-\tr(\re\,Q(z))=\frac{1}{2}\left(\re\,\zeta(z)\int_{\R}\widetilde{V}_{11}(x)\,\rd x+\re\,\zeta(z)^{-1}\int_{\R}\widetilde{V}_{22}(x)\,\rd x\right)<1,
\]
a contradiction to \eqref{eq. imaginary contradiction}. The last claim is immediate since \eqref{eq. imaginary potentials} reduces to \eqref{eq. sum less than 2} in the case $m=0$. 
\end{proof}

% \begin{remark} IM ZETA=0!!!
% Example \ref{example} with $\tau=0$ shows that Theorem \ref{thm. imaginary potential} is sharp. Indeed, in this case the delta-potentials $V_{\tau}$ in \eqref{eq. delta-potentials} are of the form

% \[
% V_0=\I\,\kappa\,\delta_0 \,I_{\C^2},\quad 0<\kappa<1.
% \]

% Recalling formula \eqref{eq. zeta delta potentials}, we see that

% \[
% \frac{z+m}{\sqrt{z^2-m^2}}=\zeta(z)=\frac{1\pm\sqrt{1-\kappa^2}}{\kappa}
% \]

% is real. An easy calculation now shows that

% \[
% \left(\re\,\frac{z+m}{\sqrt{z^2-m^2}}\right)\kappa+\biggl(\re\,\frac{\sqrt{z^2-m^2}}{z+m}\biggr)\kappa=2.
% \]

% Hence, the left hand side of \eqref{eq. imaginary potentials}, with both $L^1$-norms replaced by $\kappa$, reaches the critical value $2$.  
% \end{remark}

\section{Slowly decaying potentials}\label{section Slowly decaying potentials}

In this section we consider potentials decaying more slowly at infinity than just $V_{ij}\in L^1(\R)$ as in Theorem \ref{thm. 1}. We assume that $V_{ij}\in L^1(\R)+L^{\infty}_0(\R)$, i.e.\ there exists a decomposition $V=W+X$ such that $W_{ij}\in L^1(\R)$and $X_{ij}\in L^{\infty}_0(\R)$; here, $L^{\infty}_0(\R)$ is the space of bounded functions that vanish at infinity. Schr\"odinger operators with this type of potentials have been studied in \cite{DaNa02}.

It is well known, and easy to see, that if $V_{ij}\in L^1(\R)+L^{\infty}_0(\R)$ and $\eps>0$, then there exists a (generally non-unique) decomposition $V=W+X$ with $W_{ij}\in L^1(\R)$ and $\|X\|\leq \eps$, see \cite{DaNa02}. We set
\beq\label{Ceps}
C_{\eps}:=\inf\SSet{\int_{\R} \|W(x)\|\rd x}{V=W+X,\,W_{ij}\in L^1(\R),\, \|X\|\leq \eps}\in [0,\infty).
\eeq
%

% Again, in order to avoid technical complications we shall assume that $V$ is bounded, so that $V_{ij}\in L^1(\R)+L^{\infty}_0(\R)$. This restriction does not play a role for the eigenvalue bounds and may be omitted if the construction of Section \ref{section Construction of the perturbed operator} is used.

\begin{theorem}\label{thm. slowly decaying 1}
Let $V=(V_{ij})_{i,j=1}^2$ with $V_{ij}\in L^1(\R)+L^{\infty}_0(\R)$ for $i,j=1,2$.
Let $z\in\rho(H_0)=\C\setminus\big((-\infty,-m]\cup [m,\infty)\big)$ and  let $\eta$, $\Phi$ be defined as in \eqref{Phi, eta}, i.e.\
\beq\label{eq. etaofphi(z)}
\eta(|\Phi(z)|)=\frac{1}{\sqrt{2}}\,\sqrt{1+\frac{|z|^2+m^2}{|z^2-m^2|}},
\eeq
and $C_{\eps}$ as in \eqref{Ceps}. If for some $\eps>0$ 
\beq\label{eq. complicated spectral inclusion I}
C_{\eps}<\eta(|\Phi(z)|)^{-1}
\eeq
and 
\beq\label{eq. complicated spectral inclusion II}
\frac{1}{\dist(z,\sigma(H_0))}+\frac{\eta(|\Phi(z)|)^2}{\im\, \sqrt{z^2-m^2}} \, \frac{C_{\eps}}{1-\eta(|\Phi(z)|)\,C_{\eps}}<\frac{1}{\eps}, 
\eeq
 then $z\notin\sigma(H)$.
\end{theorem}

\begin{remark}
If $V_{ij}\in L^1(\R)$, then, in the limit $\eps\to 0$, the condition \eqref{eq. complicated spectral inclusion I} becomes \eqref{eq. v less than c} since $\lim_{\eps\to 0}C_{\eps}=\|V\|_1$ (compare \eqref{eq. Qleqmv}), and \eqref{eq. complicated spectral inclusion II} is automatically satisfied.
Hence, Theorem \ref{thm. 1} is a special case of Theorem~\ref{thm. slowly decaying 1}.
\end{remark}

\begin{proof}
Again, in order to avoid technical complications we shall assume that $V$ is bounded. This restriction does not play a role for the eigenvalue bounds and may be omitted if the construction of Section \ref{section Construction of the perturbed operator} is used.

It can be shown that the infimum in \eqref{Ceps} is in fact a minimum, see \cite{DaNa02}. Let $W$ be the corresponding minimizing element, and set $X:=V-W$. Let
\begin{alignat*}{3}
A_{W}&:=|W|^{1/2},\quad& B_{W}&:=U_{W}|W|^{1/2},\\
A_{X}&:=|X|^{1/2},\quad& B_{X}&:=U_{X}|X|^{1/2},
\end{alignat*}
where $U_W$ and $U_X$ are the partial isometries in the polar decompositions of $W$ and $X$, respectively. Set $\K:=\H\oplus\H$ and
define the operators
\beq\label{eq. factorization 2}
A:=\begin{pmatrix}A_W\\A_X\end{pmatrix}:\H\to \K,\quad B:=\begin{pmatrix}B_W \!\!&\!\! B_X\end{pmatrix}:\K\to \H.
\eeq
Then $V=BA$ and $z\in\rho(H_0)$ is an eigenvalue of $H$ if and only if $-1$ is an eigenvalue of $Q(z)$,
\[
Q(z):=AR_0(z)B=\begin{pmatrix}A_W R_0(z) B_W & A_W R_0(z) B_X\\ A_X R_0(z) B_W & A_X R_0(z) B_X\end{pmatrix},\quad z\in\rho(H_0).
\]
Since $\|R_0(z)\|=1/\dist (z,\sigma(H_0))$ and $\|A_X\|=\|B_X\|=\eps^{1/2}<\dist(z,\sigma(H_0))^{1/2}$ by \eqref{eq. complicated spectral inclusion II}, it follows that $I_{\H}+A_X R_0(z) B_X$ has a bounded inverse. 
By the well-known Schur-Frobenius factorization (see e.g.\ \cite[Proposition 1.6.2]{CT}), $I_{\K}+Q(z)$ has a bounded inverse if and only if so does its Schur complement $S(z)$,
\[
S(z):=I_{\H}+A_W R_0(z) B_W-A_W R_0(z) B_X\, (I_{\H}+A_X R_0(z) B_X)^{-1}\, A_X R_0(z) B_W.
\]
By a Neumann series argument, the latter holds whenever 
\begin{align}\label{eq. omega(z)}
\omega  (z):=\frac{\|A_W R_0(z) B_X\|\, \|A_XR_0(z) B_W\|}{(1-\|A_W R_0(z) B_W\|)(1- \|A_X R_0(z) B_X\|)}<1,
\end{align}
provided that $I_{\H}+A_W R_0(z) B_W$ has a bounded inverse as well.
By the estimates used in the proof of Theorem \ref{thm. 1}, we have
\[
\|A_W R_0(z) B_W\|\leq \eta(|\Phi(z)|)\,C_{\eps}<1
\]
by \eqref{eq. complicated spectral inclusion I}. Together with \eqref{eq. norm of AR0 and R0B} this yields
\[
\omega(z)\leq \frac{\eps\,C_{\eps}\eta(|\Phi(z)|)^2}{(\im\, \sqrt{z^2-m^2})(1-\eta(|\Phi(z)|)\,C_{\eps})\,(1-\eps/\dist(z,\sigma(H_0)))}.
\]
It is not difficult to check that the right hand hand side above is $<1$ if (and only if) \eqref{eq. complicated spectral inclusion II} holds.
\end{proof}

Theorem \ref{thm. slowly decaying 1} is the analogue of \cite[Theorem~1.5]{DaNa02} for Dirac operators. 
The next theorem is the counterpart to \cite[Theorem 2.9]{DaNa02}. Keeping the same notation as in~\cite{DaNa02}, we define the positive, decreasing convex function
\[
F_V(s):=\sup_{y\in\R}\int_{\R}\|V(x)\|\,\e^{-s|x-y|}\rd x,\quad s>0.
\]

\begin{theorem}\label{thm. slowly decaying 2}
Let $V=(V_{ij})_{i,j=1}^2$ with $V_{ij}\in L^1(\R)+L^{\infty}_0(\R)$ for $i,j=1,2$.
Let $z\in\rho(H_0)=\C\setminus\big((-\infty,-m]\cup [m,\infty)\big)$ and  let $\eta$, $\Phi$ be defined as in \eqref{eq. etaofphi(z)}. 
If 
\beq\label{eq. condition on F}
\eta(|\Phi(z)|)\, F_V\left(\,\im \, \sqrt{z^2-m^2}\right)< 1,
\eeq 
then $z\notin\sigma(H)$.
If the equation $F_V(\mu)=\mu/m$ has a solution $\mu_0\in (-m,m)$, it is unique and
\beqnt
\sigma(H)\cap \left(-\sqrt{m^2-\mu_0^2},\sqrt{m^2-\mu_0^2}\right)=\emptyset.
\eeqnt
\end{theorem}

\begin{remark}
If $V_{ij}\in L^1(\R)$, then by \cite[Lemma 2.1]{DaNa02}
\[
F_V(s)\leq F_V \|V(x)\|_1,\quad s>0.
\]
Hence, Theorem \ref{thm. 1} is a special case of Theorem \ref{thm. slowly decaying 2}.
\end{remark}

\begin{proof}
As in the proof of Theorem \ref{thm. 1}, we assume that $V$ is bounded and use the factorization $V=BA$ with $A=|V|^{1/2}$, $B=U|V|^{1/2}$ (see \eqref{eq. standard factorization}). As before, we set $Q(z)=AR_0(z)B$ (see \eqref{eq. def. Q}).
% and $k(z)=\sqrt{z^2-m^2}$. 

Using a straightforward generalization of the Schur inequality to matrix-valued kernels, we obtain
\begin{align*}
\|Q(z)\|\leq \left(\sup_{x\in\R}\int_{\R} \|Q(x,y;z)\|\, \frac{\rd y}{\rho(x,y)}\right)^{1/2}\left(\sup_{y\in\R}\int_{\R} \|Q(x,y;z)\|\,\rho(x,y)\rd x\right)^{1/2},
\end{align*}
where $Q(x,y;z)$ is the kernel of $Q(z)$ and $\rho(x,y)$ is a positive weight. Choosing $\rho(x,y):=\|V(x)\|^{1/2}\|V(y)\|^{-1/2}$ and using that $|R_0(x,y;z)|\leq \eta(|\Phi(z)|)$, we arrive at
\[
\|Q(z)\|\leq \eta(|\Phi(z)|)\,  F_V(\,\im \, \sqrt{z^2-m^2}).
\]
This proves the first part of the theorem. 

Let $z\in(-m,m)$. Observing that, by \eqref{eq. etaofphi(z)}, 
\[
\eta(|\Phi(z)|)=\frac{1}{\sqrt{2}}\,\sqrt{1+\frac{m^2+z^2}{m^2-z^2}}=\frac{m}{\sqrt{m^2-z^2}},
\]
we infer that
\[
\eta(|\Phi(z)|)\,  F_V\left(\,\im \, \sqrt{z^2-m^2}\right)=1\quad \Longleftrightarrow \quad F_V\left(\sqrt{m^2-z^2}\right)=\frac{\sqrt{m^2-z^2}}{m}. 
\]
Since the function $\mu\mapsto F_V(\mu)$ is decreasing \cite[Lemma 2.1]{DaNa02} and $\mu\mapsto\mu/m$ is increasing, the solution $\mu_0\in(-m,m)$ of the latter equation (which exists by assumption) is unique, and $F_V(\mu)<\mu/m$ for $\mu>\mu_0$. Therefore, 
\[
\eta(|\Phi(z)|)\,  F_V\left(\,\im \, \sqrt{z^2-m^2}\right)<1,\quad |z|<\sqrt{m^2-\mu_0^2},
\]
and hence $z\notin\sigma(H)$ by the first part of the Theorem.
\end{proof}

\begin{remark}
Using different factorizations of $V$, one infers from the proof of Theorem \ref{thm. slowly decaying 2} that
\[
\eta(|\Phi(z)|)\, \inf\left\{F_{A'^2}(\im\,\sqrt{z^2-m^2})^{1/2}\cdot F_{B'^2}(\im\,\sqrt{z^2-m^2})^{1/2}\right\}<1\implies  z\in\rho(H),
\]
where the infimum is taken over all factorizations $V=B'A'$.
\end{remark}

Theorem \ref{thm. slowly decaying 1} enables us to obtain eigenvalue bounds in terms of higher $L^p$-norms of the potential $V$.

\begin{corollary}\label{cor. higher Lp}
Suppose $V_{ij}\in L^p(\R)$ for $i,j=1,2$ and some $p\in(1,\infty)$, and set
\[
\|V\|_p:=\left(\int_{\R}\|V(x)\|^p \rd x\right)^{1/p}.
\]
Let $z\in\rho(H_0)=\C\setminus\big((-\infty,-m]\cup [m,\infty)\big)$ and  let $\eta$, $\Phi$ be defined as in \eqref{eq. etaofphi(z)}.
If
\beq\label{eq. Lp condition on V}
\eta(|\Phi(z)|)\,\left(\frac{2(p-1)}{p}\right)^{(p-1)/p} \, \left(\im\,\sqrt{z^2-m^2}\right)^{-(p-1)/p}\,\|V\|_p<1,
\eeq
then $z\notin\sigma(H)$.
%Here, as usual, $p'=p/(p-1)$.
\end{corollary}

\begin{proof}
This is a consequence of Theorem \ref{thm. slowly decaying 1} and the inequality
\[
F_V(s)\leq \left(\frac{2(p-1)}{p}\right)^{(p-1)/p}\!s^{-(p-1)/p}\,\,\|V\|_p \,,
\]
see \cite[Corollary 2.17]{DaNa02}.
\end{proof}

Although the conditions in the above theorems seem to be very complicated, they may still provide explicit eigenvalue bounds as the following example shows.

\begin{example}
Let $\mu\in\C$, $\re\,\mu\neq 0$, and consider the massless Dirac operator $H_{\mu}=H_0+V_{\mu}$ with potential
\[
V_{\mu}(x)=\frac{2\mu}{\sinh(2\mu x+\I)}\,\begin{pmatrix}1&0\\0&-1\end{pmatrix},\quad x\in\R, 
\]
see \cite{Syroid87}. Since
\[
\|V_{\mu}\|_p^p=(2|\mu|)^{p-1}\int_{\R} \frac{1}{|\sinh(\e^{\I\arg(\mu)}x+\I)|^{p}}\,\rd x
\]
and $\eta(|\Phi(z)|)=1$ for $m=0$ by \eqref{eq. etaofphi(z)}, Corollary \ref{cor. higher Lp} %it follows that the condition \eqref{eq. Lp condition on V} cannot hold for the eigenvalue $z=\I\,\mu$ of $H_{\mu}$.
implies that for every $p>1$, all eigenvalues of $H_{\mu}$ are contained in the strip
%
% \[
% |\im\,z|\leq |\mu|\,\frac{2(p-1)}{p}\left(\int_{\R} \frac{1}{|\sinh(x+\I)|^{p}}\,\rd x\right)^{1/(p-1)},\quad p>1.
% \]
%
\beqnt %\label{eq. Syroid example}
\sigma_{\rm d}(H_{\mu})\subset\SSet{z\in\C}{|\im\,z|\leq |\mu|\,\frac{4(p-1)}{p}\left(\int_{\R} \frac{1}{|\sinh(\e^{\I\arg(\mu)}x+\I)|^{p}}\,\rd x\right)^{1/(p-1)}}.
\eeqnt
For $p=1$, one can check that 
\[
\|V_{\mu}\|_1=\int_{\R} \frac{1}{|\sinh(\e^{\I\arg(\mu)}x+\I)|}\,\rd x \geq \int_{\R} \frac{1}{|\sinh(x+\I)|}\,\rd x \,\, (\approx 3.4184)
\]
is greater than one (and independent of $|\mu|$) so that Theorem \ref{thm. 1} cannot exclude the occurrence of nonreal eigenvalues.
In fact, it was shown in \cite{Syroid87} that $H_{\mu}$ does have the nonreal eigenvalue $\I\mu$.
%Incidentally, this together with \eqref{eq. Syroid example} implies the following inequality for arbitrary $p>1$: 
%
%\[
%\int_{\R} \frac{1}{|\sinh(x+\I)|^{p}}\,\rd x\geq\left(\frac{p}{4(p-1)}\right)^{p-1}.
%\]
%
\end{example}

\medskip
The result of Corollary \ref{cor. higher Lp} may also be used to prove that $H$ is similar to a block diagonal matrix operator if the $L^p$-norm is sufficiently small and $p\in[2,\infty]$. For more results on block-diagonalization of 
Dirac operators as well as abstract Hilbert space operators, the reader is referred to \cite{CueDiag}.

\begin{theorem}\label{thm. blockdiag}
Let $m>0$, $V_{ij}\in L^p(\R)$ for $i,j=1,2$ and some $p\in[2,\infty)$. If 
\beq\label{eq. condition on V for blockdiag}
\|V\|_p <\left(\frac{mp}{2(p-1)}\right)^{(p-1)/p},
\eeq
then $H$ is similar to a block-diagonal operator,
\[
SHS^{-1}=\begin{pmatrix}H_+&0\\0&H_-\end{pmatrix},\quad \sigma(H_{\pm})=\sigma(H)\cap\set{z\in\C}{\pm\, \re \,z>0}.
\]
\end{theorem}

\begin{proof}
% It remains to be shown that $H$ is similar to a block-diagonal operator if \eqref{eq. condition on V for blockdiag} holds.
If $z=\I\,t$, $t\in\R$, then \eqref{eq. Lp condition on V} is less than one, i.e.\
\beq\label{eq. Qit}
\|Q(\I\,t))\|<\left(\frac{2(p-1)}{p}\right)^{(p-1)/p}\!\left(\sqrt{t^2+m^2}\right)^{-(p-1)/p}\left(\frac{mp}{2(p-1)}\right)^{(p-1)/p} \leq 1,
\eeq
hence $\I\R\subset\rho(H)$. Let again $A:=|V|^{1/2}$, $B:=U|V|^{1/2}$, and set $Y:=A^p$. Since $A_{ij}\in L^{2p}(\R)$, it follows that $Y_{ij}\in L^2(\R)$, hence $Y$ is $H_0$-bounded (see for instance \cite[Satz 17.7]{Weid2}). By Heinz' inequality, $Y^{\al}$ is $|H_0|^{\al}$-bounded for any $\al\in (0,1)$. In particular, for $\al=1/p$, $A$ is $|H_0|^{1/p}$-bounded. Thus, since $|H_0|^{1/p}\geq (m)^{1/p}$, there exists a constant $\delta_m<\infty$ such that for all $ z\in\rho(H_0)$
\beq\label{eq. deltam 1}
\|AR_0(z)\|\leq \delta_m\, \||H_0|^{1/p}R_0(z)\|.
\eeq
Analogously, one can show that
\beq\label{eq. deltam 2}
\|R_0(z)B\|\leq \delta_m\, \||H_0|^{1/p}R_0(z)\|.
\eeq
For $\chi\in\C$, $|\chi|<1$, let $H(\chi):=H_0+\chi V$. By inspection of the resolvent of $H(\chi)$,
\[
(H(\chi)-z)^{-1}=R_0(z)-\chi R_0(z)B\left(I_{\K}+\chi Q(z)\right)^{-1}AR_0(z),
\]
it is easily seen that $H(\chi)$, $|\chi|<1$, is a holomorphic family. For $f\in\H$, we define
\beq\label{eq. projection}
P(\chi)f:=\frac{1}{2}\,f+\frac{1}{2\pi}\lim_{R\to\infty}\int_{-R}^R (H(\chi)-\I t)^{-1}f\rd t, \quad |\chi|<1.
\eeq
We shall show that the limit exists and that $P(\chi)$ is a bounded-holomorphic family of projections.
By \cite[II.4.2]{Ka}, it then follows that there exists a bounded-holomorphic family of isomorphisms $U(\chi)$ 
such that
\[
U(\chi)P(\chi)U(\chi)^{-1}=P(0),\quad \chi\in\C,\quad |\chi|<1.
\]
On the other hand, by the standard Foldy-Wouthuysen transformation (i.e.\ diagonalizing $H_0$ in momentum space, see e.g.\ \cite{Th}), there exists a unitary operator $\widetilde{U}$ such that
\[
\widetilde{U}P(0)\widetilde{U}^{-1}=\begin{pmatrix}1&0\\0&0\end{pmatrix}.
\]
The claim thus follows with $S:=\widetilde{U}U(1)$.

Since $H_0$ is selfadjoint, the right hand side of \eqref{eq. projection} exists for $\chi=0$ and coincides with the spectral projection onto the positive spectral subspace of $H_0$, by the spectral theorem. It is thus sufficient to show the convergence of the integral
\[
\lim_{R\to\infty}\int_{-R}^R \left(\left(H(\chi)-\I t)^{-1}-R_0(\I t)\right)f,g\right)\rd t
\]
uniformly in $g\in\H$, $\|g\|=1$, and locally uniformly in $\chi\in\C$, $|\chi|<1$. Indeed, since by \eqref{eq. Qit},
\[
q_0:=\sup_{t\in\R}\|Q(\I t)\|<1,
\]
the estimates \eqref{eq. deltam 1}, \eqref{eq. deltam 2} imply, for $|\chi|<1$,
\begin{align*}
&\int_{-R}^R \left|\left(\left(H(\chi)-\I t)^{-1}-R_0(\I t)\right)f,g\right)\right|\rd t\\
&\leq(1-q_0)^{-1}\int_{-R}^R \|AR_0(\I t)f\|\,\|R_0(\I t)Bg\|\rd t\\
&\leq(1-q_0)^{-1}\int_{-R}^R \||H_0|^{1/p}R_0(\I t)f\|\, \||H_0|^{1/p}R_0(\I t)g\|\rd t\\
&\leq(1-q_0)^{-1}\left(\int_{-R}^R \||H_0|^{1/p}R_0(\I t)f\|^2\rd t\right)^{1/2}\left(\int_{-R}^R \||H_0|^{1/p}R_0(\I t)g\|^2\rd t\right)^{1/2}.
\end{align*}
Denoting by $E(\cdot)$ the spectral function of $H_0$, we can estimate
\begin{align*}
\int_{-R}^R \||H_0|^{1/p}R_0(\I t)f\|^2\rd t&\leq \int_{\sigma(H_0)}\int_{-\infty}^{\infty}\frac{|s|^{2/p}}{s^2+t^2}\,\rd t\rd\|E(s)f\|^2\\
&=\pi\int_{\sigma(H_0)}|s|^{(2/p)-1} \rd\|E(s)f\|^2
\leq \pi (m)^{(2/p)-1}\,\|f\|^2.
\end{align*}

The fact that $P(\chi)$ is a spectral projection corresponding to the right half plane may be deduced from \cite[Theorem 3.1]{GGK1} in combination with the residue theorem, see also \cite[Theorem 1.1]{LT}, \cite[Theorem 2.4]{CueDiag}. In order to apply the latter, it  remains to be shown that 
\beq\label{eq. resolvent norm to zero}
\lim_{t\to\infty}\|(H-\I t)^{-1}\|=0.
\eeq
By the spectral theorem for $H_0$, 
\begin{align*}
\|(H-\I t)^{-1}\|&\leq \|(H_0-\I t)^{-1}\|+\|(H-\I t)^{-1}-(H_0-\I t)^{-1}\|\\
&\leq \frac{1}{|t|}+(1-q_0)^{-1}\||H_0|^{1/p}R_0(\I t)\|^2\leq \frac{1}{|t|}+\frac{C}{|t|^{1-1/p}}
\end{align*}
for some $C>0$. This proves \eqref{eq. resolvent norm to zero}.
\end{proof}

\begin{remark}
Similar estimates as in \eqref{eq. Lp condition on V} have been derived in \cite{CueTre12} by a more abstract approach. For example, for $m>0$ and $p=2$, the results of \cite{CueTre12} imply that
\beq\label{eq. spec. incl. CueTre12}
 \sigma(H)\subset\SSet{z\in\C}{|\im\, z|\leq 2\,\|V\|_2^2\,(1+|z|)^{1/2}}.
\eeq
In comparison, \eqref{eq. Lp condition on V} above implies that  
\beq\label{eq. compare to spec. incl. CueTre12}
\sigma(H)\subset\SSet{z\in\C}{\im\sqrt{z^2-m^2}\leq \eta(|\Phi(z)|)^2 \|V\|_2^2}.
\eeq
Asymptotically, \eqref{eq. spec. incl. CueTre12} and \eqref{eq. compare to spec. incl. CueTre12} yield that for $z\in\sigma(H)$
\[
|\im\,z|\leq 2 \,\|V\|_2^2 \,|z|^{1/2}\quad \mbox{and}\quad |\im\,z|\leq \,\|V\|_2^2,\quad |z|\to\infty,
\]
respectively. The second estimate is clearly superior, which is not surprising since the results of \cite{CueTre12} are of much more general nature. They are applicable to Dirac operators in arbitrary dimension as well as to abstract Hilbert space operators.
\end{remark}

\section{Embedded eigenvalues and resonances}\label{section Embedded eigenvalues and resonances}

In this section we show how the previous results may be applied to locate the embedded eigenvalues and resonances of Dirac operators with Hermitian potentials using the method of complex scaling. To this end, we assume that $V$ is dilation-analytic.

For simplicity, we restrict ourselves to the case $\|V\|_1<1$ (see Theorem \ref{thm. 1}); the formulation and proof of analogous results using Theorems \ref{thm. slowly decaying 1} and \ref{thm. slowly decaying 2} is straightforward. 
Moreover, we use a boundedness assumption on $V$ which can be relaxed using the construction of the extension $H$ of $H_0+V$ in Section~\ref{section Construction of the perturbed operator}.

Let $U(\theta)$ be the unitary dilation in $L^2(\R)\otimes\C^2$, given by
\[
(U(\theta)f)(x):=\e^{\theta/2}f(\e^{\theta}x),\quad x,\theta\in\R.
\]
For $\al\in(0,\pi/2)$ let $\Sigma_{\al}:=\set{z\in\C\setminus\{0\}}{|\arg(z)|<\al}$, where $-\pi<\arg(z)<\pi$.

%We now state our hypothesis on the potential $V$ for this section. 

\begin{hypothesis}\label{hypothesis V resonance}
Assume that there exists $\al\in(0,\pi/2)$ such that:% the following hold:
\begin{itemize}
\item[\rm i)] $V:\Sigma_{\al}\cup(-\Sigma_{\al})\to\C^{2\times 2}$ is a bounded analytic function;
\item[\rm ii)] The restriction of $V$ to the real axis is Hermitian-valued;
\item[\rm iii)] For each $\beta\in (0,\al)$ the functions $V(\e^{\I\,\varphi}\cdot)$, $|\varphi|\leq \be$, are in $L^1(\R,\C^{2\times 2})$ with uniformly bounded $L^1$-norms.
\end{itemize}
\end{hypothesis}

We define the complex-dilated operators
\begin{align*}
H_0(\theta)&:=U(\theta)H_0 U(\theta)^{-1}=-\I\e^{-\theta}\frac{\rd}{\rd x}\sigma_1+m\sigma_3,\\
V(\theta)&:=U(\theta)V U(\theta)^{-1}=V(\e^{\theta}\cdot),\\[1mm]
H(\theta)&:=U(\theta)(H_0+V)U(\theta)^{-1}=H_0(\theta)+V(\theta).
\end{align*}
It is straightforward to check that $H_0(\theta)$ has an extension to an entire family of type (A) in the sense of Kato \cite[VII.2]{Ka}, see e.g.\ \cite[Lemma 1]{We73}.

\begin{proposition}\label{prop. resonances}
Assume that Hypothesis \rref{hypothesis V resonance} is satisfied for some $\al\in(0,\pi/2)$. Then the following hold:
\begin{itemize}
\item[\rm i)] $V(\theta)$ has an extension to an analytic bounded operator-valued function in the strip $S_{\al}:=\set{\theta\in\C}{|\im \,\theta|<\alpha}$;
\item[\rm ii)] For $\mu\in\R$, $|\mu|$ sufficiently large, $\I\,\mu\in\rho(H(\theta))$ for all $\theta\in S_{\al}$, and for $\I\mu\in\rho(H(\theta))$ fixed,  $(H(\theta)-\I\,\mu)^{-1}$ is an analytic bounded operator-valued function in $S_{\al}$;
\item[\rm iii)] $U(\varphi)H(\theta)U(\varphi)^{-1}=H(\theta+\varphi)$ for all $\varphi\in\R$, $\theta\in S_{\al}$;
\item[\rm iv)] $\sigma(H(\theta))$ depends only on $\im\,\theta$;
\item[\rm v)] $\sigma_{\rm e}(H_0(\theta))=\set{\pm\sqrt{\e^{-2\theta}p^2+m^2}}{p\in\R}$;
\item[\rm vi)] $\sigma_{\rm d}(H(\theta))\cap\R=\sigma_{\rm p}(H)\setminus\{-m,m\}$;
\item[\rm vii)] For $\im\,\theta\in (0,\al)$, all nonreal eigenvalues of $H(\theta)$ lie in the region
\[
D_{\theta}:=\set{\pm\sqrt{\e^{-2\omega}p^2+m^2}}{p\in\R,\, \im\,\omega\in[0,\im\,\theta]},
\]
see Fig.\ {\rm 2}. If $\,0<\im\,\theta_1<\im\,\theta_2<\al$, then $\sigma_{\rm d}(H(\theta_1))\subset \sigma_{\rm d}(H(\theta_2))$.
\item[\rm viii)] For $\beta\in (0,\al)$, the function $\varphi\mapsto\|V(\e^{\I\varphi}\cdot)\|_1$ is logarithmically convex in the interval $[-\be,\be]$.
\end{itemize}
\end{proposition}

\begin{figure}[ht]
\begin{pspicture*}(-5,-2.5)(5,2.5)
\psline{<->}(-5,0)(5,0)
\psline{<->}(0,-2.5)(0,2.5)
\psplot[plotpoints=200,%
plotstyle=curve,linecolor=red,fillstyle=vlines,fillcolor=lightgray]%
{1}{10}
{
x x mul 1 sub sqrt -.5 mul 
}
\psline[linestyle=none,fillstyle=vlines,fillcolor=lightgray](1,0)(7,0)(7,-3.5)(1,0)
\psplot[plotpoints=200,%
plotstyle=curve,linecolor=red,fillstyle=vlines,fillcolor=lightgray]%
{-1}{-10}
{
x x mul 1 sub sqrt .5 mul 
}
\psline[linestyle=none,fillstyle=vlines,fillcolor=lightgray](-1,0)(-7,0)(-7,3.5)(-1,0)
\rput(1,.2){$m$}
\rput(-1,-.2){$-m$}
\rput(-3,2){$\textcolor{red}{\sigma_{\rm e}(H(\theta))}$}
\rput(3,-2){$\textcolor{red}{\sigma_{\rm e}(H(\theta))}$}
\rput(-4,.5){\colorbox{white}{$D_{\theta}$}}
\rput(4,-.5){\colorbox{white}{$D_{\theta}$}}
\rput(3.5,-1){\psdot}
\rput(1.4,-.2){\psdot}
\rput(3,-.6){\psdot}
\rput(-2,.4){\psdot}
\rput(-4,1){\psdot}
\rput(.2,0){\psdot}
\rput(.6,0){\psdot}
\rput(-.5,0){\psdot}
\rput(-.5,0){\psdot}
\rput(-3,0){\psdot}
\rput(2,0){\psdot}
\rput(4,0){\psdot}
\end{pspicture*}
\caption{Eigenvalues of $H$ and the set $D_{\theta}$ enclosing resonances of $H$.}
\end{figure}
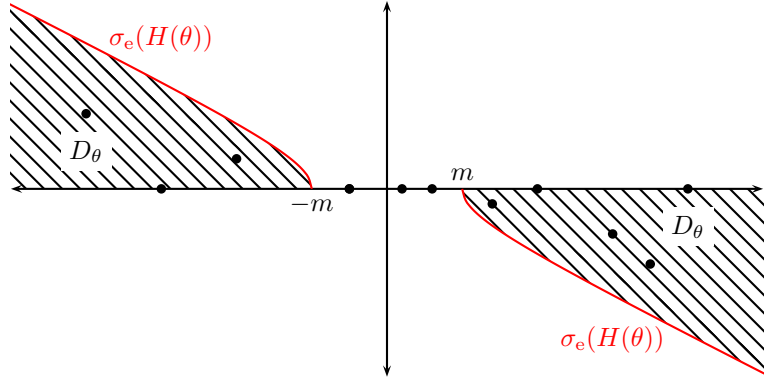

\begin{proof}%[Proof of Proposition \rref{prop. resonances}]
 i) Since $S_{\al}$ is mapped onto $\Sigma_{\al}$ under the mapping $\theta\mapsto \e^{\theta}$, it follows that $V(\theta)\in\LL(\H)$. 
It is easy to see that $V(\theta)$, $\theta\in S_{\al}$, is weakly analytic, and hence analytic in norm, see e.g.\ \cite[Theorem III.1.3.7]{Ka}.

ii) Since $V(\theta)$ is uniformly bounded in the operator norm, $\|V(\theta)\|\leq M<\infty$, the spectrum of $H(\theta)$ is contained in the $M$-neighbourhood of $\sigma(H_0(\theta))$ by the stability of bounded invertibility. Hence, $\I\,\mu\in\rho(H(\theta))$ for $|\mu|$ sufficiently large. The analyticity of $(H(\theta)-\I\,\mu)^{-1}$ follows from the formula
\[
(H(\theta)-\I\,\mu)^{-1}=(H_0(\theta)-\I\,\mu)^{-1}(I+V(\theta)(H_0(\theta)-\I\,\mu)^{-1})^{-1}
\]
and from the observation that $H_0(\theta)$ is a normal operator, whence for $|\mu|$ sufficiently large, $$\|(H_0(\theta)-\I\,\mu)^{-1}\|=\dist(\I\,\mu,\sigma(H_0(\theta))<1/M.$$ 

iii) is clearly valid for real $\theta$, and since both sides of the equation are analytic, the claim follows from the identity theorem.
iv) is a direct consequence of iii).

For the proof of v)-vii), we refer to \cite[Theorem 1]{Seba88}, compare also \cite[XIII.36]{RS4}. Unlike in \cite{Seba88}, we do not assume that $V$ is $H_0$-compact; however, as already mentioned in the introduction, since $V$ decays at infinity the resolvent difference of $H$ and $H_0$ is compact and thus their essential spectra are the same by \cite[Theorem~IX.2.4]{EE}. Since
\[
(H(\theta)-z)^{-1}-(H_0(\theta)-z)^{-1}=U(\theta)((H-z)^{-1}-(H_0-z)^{-1})U(\theta)^{-1},
\]
the same applies to the essential spectra of $H(\theta)$ and $H_0(\theta)$ and thus the proof of \cite[Theorem 1]{Seba88} carries through in the case considered here. 

viii) Let $g\in L^{\infty}(\R)$. Then 
\[
\int_{\R}V_{ij}(\e^{\theta}x)g(x)\,\rd x
\]
depends analytically on $\theta\in S_{\al}$ since on any compact subset $K\subset S_{\al}$ the absolute value of the integral is bounded by
\[
\rho\cdot\sup_{|\varphi|\leq\be}\|V(\e^{\I\,\varphi}\cdot)\|_1\cdot \|g\|_{\infty}\quad \mbox{where}\quad\rho:=\min_{\theta\in K}\e^{-\re\,\theta},\quad \be:=\max_{\theta\in K}|\im\,\theta|.
\]
Hence, the map $(\theta\mapsto V(\e^{\theta}\cdot)):S_{\al}\to L^1(\R,\C^{2\times 2})$ is weakly (and hence strongly) analytic. For $\beta\in (0,\al)$ consider the map 
\[
F:S_{\be}\to L^1(\R,\C^{2\times 2}),\quad F(\theta):=\e^{\theta}V(\e^{\theta}\cdot)
\]
which is analytic, continuous up to the boundary of $S_{\be}$, and uniformly bounded in $\overline{S_{\be}}$. The claim follows by applying Hadamard's three-lines theorem for analytic functions with values in a Banach space, see e.g.\ \cite[III.14]{DS1}, to $F$ and noting that $\|F(\I\,\varphi)\|_1=\|V(\e^{\I\,\varphi}\cdot)\|_1$.
\end{proof}

It may be shown, see \cite[Theorem 2]{Seba88}, that the resolvent $(H-z)^{-1}$ has a (many-sheeted) analytic continuation 
to the set $\rho(H_{\theta})$. The poles of the analytically continued resolvent are called the \emph{resonances} of $H$, and they are located precisely at the eigenvalues of $H_{\theta}$. We denote the set of resonances of $H$ by $\mathcal{R}(H)$. 
%The possible locations of eigenvalues and resonances of the perturbed Dirac operator is depicted in Figure 2.

\begin{theorem}\label{thm. resonances}
Assume that $V$ satisfies Hypothesis \rref{hypothesis V resonance} with $\al\in(0,\pi/2)$.
\begin{itemize}
\item[\rm i)] If $\im\,\theta\in[0,\al)$ and
\begin{align*}%\label{eq. vtheta}
v_{\theta}&:=\inf_{\im\,\theta\leq\varphi<\al}\|V(\e^{\I\,\varphi}\cdot)\|_1<1,
\end{align*}
then the resonances of $H$ satisfy the inclusion
\beq\label{eq. resonance inclusion}
\mathcal{R}(H)\cap D_{\theta}\subset K_{mr_{\theta}}(mx_{\theta})\,\dotcup \,K_{mr_{\theta}}(-mx_{\theta})
\eeq
where
\beq\label{eq. xtheta rtheta}
x_{\theta}:=\sqrt{\frac{v_{\theta}^4-2 v_{\theta}^2+2}{4(1-v_{\theta}^2)}+\frac{1}{2}}, \quad r_{\theta}:=\sqrt{\frac{v_{\theta}^4-2 v_{\theta}^2+2}{4(1-v_{\theta}^2)}-\frac{1}{2}}.
\eeq
\item[\rm ii)] If $V$ is scalar-valued and sign-definite, then $v_{\theta}=\|V(\e^{\I\im\,\theta}\cdot)\|_1$ in {\rm i)}.
\item[\rm iii)] Assume that $\|V\|_1<1$. 
Then all eigenvalues of $H$ {\rm(}including the embedded ones{\rm)} are contained in the intervals
\beq\label{eq. embedded eigenvalues}
\big(-m(x_0+r_0),-m(x_0-r_0)\big)\,\dotcup\, \big(m(x_0-r_0),m(x_0+r_0)\big),
\eeq
where $x_0$, $r_0$ are given in \eqref{eq. x0 r0} {\rm(}i.e.\ \eqref{eq. xtheta rtheta} with $v_{\theta}=v_0=\|V\|_1)$.

\item[\rm iv)] If $m=0$ and $\|V\|_1<1$, then there are no resonances close to the real axis; more precisely,
if we set
\begin{align*}
\varphi_0:=\sup\set{\im\,\theta\in[0,\al)}{v_{\theta}<1}\,(>0),
\end{align*}
then
\[
\mathcal{R}(H)\cap\Set{\pm\sqrt{\e^{-2\omega}p^2+m^2}}{p\in\R,\, \im\,\omega\in[0,\varphi_0]}=\emptyset.
\]
\end{itemize}
\end{theorem}

% \begin{remark}
% Notice that the spectral inclusion in ii) also applies to eigenvalues which are embedded in $(-\infty,-m]\cup[m,\infty)$. 
% \end{remark}

\begin{proof}
i) Let $(\theta_n)_{n\in\N}\subset S_{\al}$ be such that $\varphi_n:=\im\,\theta_n\geq \im\,\theta$, $n\in\N$, and $$\|V(\e^{\I\,\im\theta_n}\cdot)\|_1 \longrightarrow  v_{\theta},\quad n\to\infty.$$ Then there exists $N\in\N$ such that $\|V(\e^{\I\,\im\theta_n}\cdot)\|_1<1$ for all $n\geq N$.
Since
\[
\e^{\I\,\varphi_n}H(\I\,\varphi_n)=-\I\frac{\rd}{\rd x}\sigma_1+m\e^{\I\,\varphi_n}\sigma_3+\e^{\I\,\varphi_n}V(\e^{\I\,\varphi_n}\cdot)
\]
and $|\e^{\I\,\varphi_n}|=1$, Theorem \ref{thm. 1} and Proposition \ref{prop. resonances} iii) imply that for all $n\geq N$,
\beq\label{spectral inclusion for Htheta 2}
\sigma_{\rm d}(\e^{\I\,\varphi_n}H(\theta_n)\subset K_{mr_{\theta_n}}(m\e^{\I\,\varphi_n}x_{\theta_n})\,\dotcup \,K_{mr_{\theta_n}}(-m\e^{\I\,\varphi_n}x_{\theta_n}).
\eeq
In fact, we have to modify the proof of Theorem \ref{thm. 1} slightly to take the complex mass term $m':=m\e^{\I\,\theta}$ into account.
It is easy to see, however, that we only have to replace $m'r_{\theta}$ by $|m'|r_{\theta}$.
By Proposition \ref{prop. resonances} vii) and \eqref{spectral inclusion for Htheta 2}, it follows that for all $n\geq N$
\[
\mathcal{R}(H)\cap D_{\theta}\subset K_{mr_{\theta_n}}(mx_{\theta_n})\,\dotcup \,K_{mr_{\theta_n}}(-mx_{\theta_n}).
\]
Letting $n\to\infty$ proves \eqref{eq. resonance inclusion}.

ii) Without loss of generality, assume that $V\geq 0$. We show that $\|V(\e^{\I\,\varphi}\cdot)\|_1$ achieves a global minimum at $\varphi=0$. It then follows from Proposition \ref{prop. resonances} viii) that $\|V(\e^{\I\,\varphi}\cdot)\|_1$ must be monotonically increasing in $|\varphi|$, and hence $v_{\theta}=\|V(\e^{\I\im\,\theta}\cdot)\|_1$.

Let $\varphi\in(0,\al)$, $R>0$, and define the curves
\begin{align*}
\gm_1&:=\set{x\e^{\I\,\varphi}}{-R\leq x\leq R},\\
\gm_2&:=\set{R\e^{\I\,t}}{\varphi\geq t\geq 0},\\
\gm_3&:=\set{R\e^{\I\,t}}{\pi\leq t\leq \pi+\varphi}.
\end{align*}
By Cauchy's theorem,
\beq\label{eq. Cauchy}
\int_{-R}^R V(x)\,\rd x=\sum_{j=1}^3\int_{\gm_j} V(z)\, \rd z.
\eeq
We show that the contribution of the integrals over $\gm_2$ and $\gm_3$ vanishes in the limit $R\to\infty$. Indeed, let $z=R\e^{\I\,t}$, $0\leq t\leq \varphi$, be a parametrization of $\gm_2$, so that
\beq\label{eq. gamma2 contour}
\int_{\gm_2} V(z)\, \rd z=\I\,R\,\int_{0}^{\varphi} V(R\e^{\I\,t})\,\e^{\I\,t}\, \rd t.
\eeq
By Fubini's theorem and assumption iii) of Hypothesis \rref{hypothesis V resonance},
\begin{align*}
\int_{-\infty}^{\infty}\left|\int_{0}^{\varphi} V(R\e^{\I\,t})\e^{\I\,t}\, \rd t\right|\,\rd R
&\leq\int_{-\infty}^{\infty}\int_{0}^{\varphi} |V(R\e^{\I\,t})|\, \rd t\,\rd R\\ 
&= \int_{0}^{\varphi}\int_{-\infty}^{\infty} |V(R\e^{\I\,t})| \,\rd R \, \rd t\\
&=\int_{0}^{\varphi}\|V(\e^{\I\,t}\cdot)\|_1 \, \rd t\leq \varphi\cdot\sup_{|t|\leq\varphi}{\|V(\e^{\I\,t}\cdot)\|_1}.
\end{align*}
It follows that the function 
\[
R\mapsto \int_{0}^{\varphi} V(R\e^{\I\,t})\e^{\I\,t}\, \rd t
\]
belongs to $L^1(\R)$ and is thus ${\rm o}(1/R)$. Hence, by \eqref{eq. gamma2 contour}, the integral over $\gm_2$ tends to zero as $R\to\infty$. 
The proof for $\gm_3$ is analogous. 

It now follows from \eqref{eq. Cauchy} that, in the limit $R\to\infty$,
\[
\int_{\R} V(x)\,\rd x=\int_{\gm_1} V(z)\, \rd z=\e^{\I\,\varphi}\int_{\R}V(\e^{\I\,\varphi}x)\,\rd x.
\]
Taking the absolute value on both sides proves that $\|V\|_1\leq \|V(\e^{\I\,\varphi}\cdot)\|_1$ for all $\varphi\in(0,\al)$. The proof for $\varphi\in[-\al,0)$ is analogous.

iii) By the proof of Proposition \ref{prop. resonances} viii), $\|V(\e^{\I\,\varphi}\cdot)\|_1$ is continuous, so that
\[
\lim_{\varphi\searrow 0}\|V(\e^{\I\,\varphi}\cdot)\|_1=\|V\|_1.
\]
Let $(\theta_n)_{n\in\N}\subset S_{\al}$ be such that $\varphi_n:=\im\,\theta_n\to 0$ and $\|V(\e^{\I\,\varphi_n}\cdot)\|_1\to \|V\|_1$, $n\to\infty$. Moreover, let $N\in\N$ be such that $\|V(\e^{\I\,\varphi}\cdot)\|_1<1$, $n\geq N$.
If $\lm\in\R\setminus\{\pm m\}$ is an eigenvalue of $H$, then by Proposition \ref{prop. resonances} vi), $\lm\in\sigma(H(\theta_n))$ for all $n\geq N$. The inclusion \eqref{eq. embedded eigenvalues} now follows from \eqref{spectral inclusion for Htheta 2} if we take $n\to\infty$. 

iv) is immediate from i) since then $mr_{\theta}=0$ (recall that we use the convention $K_0(z_0)=\emptyset$).
\end{proof}

\begin{remark}
The resonance enclosure \eqref{eq. resonance inclusion} in Theorem \ref{thm. resonances} may be used for every~$\theta$, with $v_{\theta}<1$. However, increasing $\im\,\theta$ in order to enlarge the set $D_{\theta}$ revealing the resonances increases the size of the resonance-enclosing disks $K_{mr_{\theta}(\pm mx_{\theta})}$. For every~$\theta$, the disks $K_{mr_{\theta}(\pm mx_{\theta})}$
intersect the boundary $\sigma_{\rm e}(H(\theta))$ of $D_{\theta}$ in only one point each. The set of intersection points consists of two curves parametrized by $\im\,\theta$, and all resonances in $D_{\al}$ in the lower (upper) half plane lie below (above) these curves.
The shape of the resonance-enclosing set corresponding to Example~\ref{ex. resonances} is illustrated in Figure \ref{fig. resonance plot}.
%varying $\theta$, all these points yield two curves, so that the resonances lie in the area between the real axis and these curves. 
\end{remark}

\begin{example}\label{ex. resonances}
Consider the resonances and embedded eigenvalues for the potential
\[
V(x)=a\,\e^{-b \,x^2}\,I_{\C^2}
\]
with $a\in\R$, $b>0$. Clearly, $V$ has an analytic continuation to an entire function, bounded on $\overline{\Sigma_{\pi/4}}$. Moreover, for $|\varphi|<\pi/4$, the function $V(\e^{\I\,\varphi}\cdot)$ is in $L^1(\R)$ with norm
\[
\|V(\e^{\I\,\varphi}\cdot)\|_1=\frac{|a|\sqrt{\pi}}{\sqrt{b\,\cos(2\varphi)}},
\]
hence it is uniformly bounded for $|\varphi|\leq\be<\pi/4$. Since $V(x)\geq 0$, $x\in\R$, by Theorem \ref{thm. resonances} ii), $v_{\theta}=\|V(\e^{\I\,\im\,\theta}\cdot)\|_1$. Hence, if $|a|\sqrt{\pi}/\sqrt{b}<1,$ then $v_{\theta}<1$ for all $\theta\in[0,\pi/4)$ with
\[
\im\,\theta<\frac{1}{2}\arccos\left(\frac{|a|^2\pi}{b}\right).
\]
Therefore, for these $\theta$, Theorem \ref{thm. resonances} i) and iii) apply; for example, the resonances in $D_{\pi/6}$ lie in the union of the two disks
%shows that the resonances of $H$ satisfy $\mathcal{R}(H)\cap D_{\theta}\subset K_{mr_{\theta}}(mx_{\theta})\,\dotcup 
$K_{mr_{\pi/6}}(\pm m x_{\pi/6})$ with
\[
x_{\pi/6}=\frac{b-a^2\pi}{\sqrt{b(b-2 a^2\pi)}},\quad r_{\pi/6}=\frac{a^2\pi}{\sqrt{b(b-2 a^2\pi)}},
\]
the eigenvalues of $H$ (including the embedded ones) lie in the two intervals
\[
\left(-m\Big(1-\frac{a^2\pi}{b}\Big)^{-1/2}\!,\,-m\Big(1-\frac{a^2\pi}{b}\Big)^{1/2}\right)\dotcup\left(m\Big(1-\frac{a^2\pi}{b}\Big)^{1/2}\!,\,m\Big(1-\frac{a^2\pi}{b}\Big)^{-1/2}\right).
\]
Figure \ref{fig. resonance plot} shows the region of resonance enclosure in the lower half plane; the picture in the upper half plane is just the mirror image.
\end{example}

\begin{figure}[h]\label{fig. resonance plot}
\begin{center}
\includegraphics[scale=0.4]{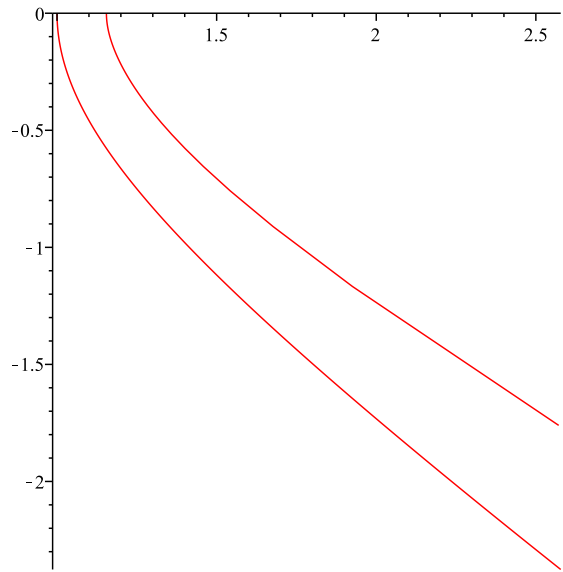}
\end{center}
\caption{The resonances of Example \ref{ex. resonances} in the lower half plane are are situated within the area between the two red curves.}
\end{figure}

%%%%%%%%%%%%%%%%%%%%%%%%%%%%%%

\section{Construction of $H$ for potentials in $L^1(\R)+L^{\infty}_0(\R)$}\label{section Construction of the perturbed operator}

In Sections \ref{section Integrable potentials}--\ref{section Embedded eigenvalues and resonances} we assumed in all proofs that $V$ is bounded, so that we could conveniently define the sum of $H_0$ and $V$. %it is obvious from the proofs that the previous estimates only depend on the decay properties of $V$.
In this final section we show how to construct a closed extension $H$ of $H_0+V$ for $V\in L^1(\R)+L^{\infty}_0(\R)$.

One might first try to approximate $V\in L^1(\R)+L^{\infty}_0(\R)$ by bounded potentials $V_n$, and then show that the operators $H_n=H_0+V_n$ converge in the norm-resolvent topology to some operator $H$. If $V$ were Hermitian-valued (and thus $H_n$, $H$ selfadjoint), we could conclude that the eigenvalue estimates also hold for the limit operator $H$. However, for non-Hermitian potentials, this need not be true since the spectrum is not lower-semicontinuous on the metric space of closed operators, see~\cite[IV.3.2]{Ka}.

Therefore, we need a more direct access to the perturbed operator $H$. If we define it via its resolvent by equation \eqref{eq. resolvent formula bounded V}, then it will turn out to be a closed extension of $H_0+V$. The precise statement is given in the subsequent abstract theorem, which includes the general version of the Birman-Schwinger principle.
We note that this construction is more general than a quadratic form approach or even an operator perturbation approach, see \cite[Remark 2.4 iii)]{GLMZ05}.

\begin{theorem}\label{abstract theorem}
Let $\H$ and $\K$ be Hilbert spaces, and let $H_0:\H\to\H$, $A:\H\to\K$ and $B:\K\to\H$ be closed densely defined operators. Suppose that
$\rho(H_0)\neq \emptyset$ and that the following hold:
\begin{itemize}
\item[\rm a)] $AR_0(z)\in\LL(\H,\K)$ and $R_0(z)B$ has bounded closure.%\\[1mm]
\item[\rm b)] For some {\rm(}and hence for all {\rm)} $z\in\rho(H_0)$, the operator $AR_0(z)B$ has bounded closure
\[
Q(z):=\overline{AR_0(z)B}\in\LL(\K). %\\[1mm] 
\]

\item[\rm c)] $-1\in\rho(Q(z_0))$ for some $z_0\in\rho(H_0)$.%\\[1mm]
\end{itemize}
Then there exists a closed densely defined extension $H$ of $H_0+BA$ whose resolvent~$R(z)=(H-z)^{-1}$, $z\in\rho(H)$, is given by 
\beq\label{resolvent factorized perturbation}
R(z)=R_0(z)-\overline{R_0(z)B}\left(I_{\K}+Q(z)\right)^{-1}\!AR_0(z)\in\LL(\H),\quad z\in\rho(H_0)\cap\rho(H),
\eeq 
with
\[
\rho(H)\cap\rho(H_0)=\set{z\in\rho(H_0)}{-1\in\rho(Q(z))}.
\]
Moreover, for $z\in\rho(H_0)$, the subspaces $\ker(H-z)$ and $\ker(I+Q(z))$ are isomorphic.
\end{theorem}

\begin{proof}
The proof may be found e.g.\ in \cite{GLMZ05}, compare also \cite{Ka65, Ka83}.
\end{proof}

\begin{remark}
If $H_0+V$ has nonempty resolvent set, and is, hence, closed, then $H=H_0+V$. In particular, this is the case whenever $V$ is bounded, or more generally, $H_0$-bounded with relative bound less than one. For example, this holds if $V_{i,j}\in L^p(\R)$ for some $p\in[2,\infty]$, see e.g.\ \cite[Satz 17.7]{Weid2}. Note that the whole $L^p$-scale, $p\in[1,\infty]$, is contained in the class $L^1(\R)+L_0^{\infty}(\R)$ considered in Section~\ref{section Slowly decaying potentials}.
\end{remark}

Since the proofs of Sections \ref{section Integrable potentials}--\ref{section Embedded eigenvalues and resonances} only involve the resolvent $R_0(z)$, they admit straightforward generalizations to the case where $V$ is unbounded and $H$ is the operator given by Theorem \ref{abstract theorem}; one just has to replace $R_0(z)B$ and $AR_0(z)B$ by their bounded closures everywhere. Indeed, \eqref{eq. Qleqmv} and \eqref{eq. norm of AR0 and R0B} guarantee that the conditions a)--c) of Theorem~\ref{abstract theorem} are satisfied. What remains to be shown is that

\begin{enumerate}
\item the different factorizations of $V$ used in Section~\ref{section Slowly decaying potentials} lead to the same extension~$H$;
\item we still have $\sigma_{\rm e}(H)=\sigma_{\rm e}(H_0)$.
\end{enumerate}
To address (1) we introduce the following definition.

\begin{definition}\label{def. compatible}
Let $\H$, $\K$, $\K'$ be Hilbert spaces, and let $H_0:\H\to\H$, $A:\H\to\K$, $B:\K\to\H$, $A':\H\to\K'$, $B':\K'\to\H$ be such that $BA=B'A'$. Suppose that the triples $(H_0,A,B)$ and $(H_0,A',B')$ satisfy the assumptions of Theorem \rref{abstract theorem}. The two factorizations $V:=BA=B'A'$ are called \emph{compatible} if the following hold:
\begin{itemize}
\item[\rm i)] The operators $A'R_0(z)B$ and $AR_0(z)B'$ have bounded closure for one {\rm(}and hence for all{\rm)} $z\in\rho(H_0)$,
\[
F(z):=\overline{A'R_0(z)B}\in\LL(\K,\K'),\quad G(z):=\overline{AR_0(z)B'}\in\LL(\K',\K).
\]
\item[\rm ii)] There exist dense linear manifolds $\mathcal{C}\subset\H$, $\dom\subset\K$ and $\dom'\subset\K'$ such that for all $z\in\rho(H_0)$,
\begin{align*}
\mathcal{C}&\subset\set{f\in\H}{R_0(z)f\in\dom(V),\,R_0(z)VR_0(z)f\in\dom(V)},\\
\dom&\subset\set{f\in\dom(B)}{R_0(z)Bf\in\dom(V)},\\
\dom'&\subset\set{f\in\dom(B')}{R_0(z)B'f\in\dom(V)}. 
\end{align*}
\end{itemize}
\end{definition}

\begin{proposition}\label{corollary independend of factorization}
If $V=BA=B'A'$ are two compatible factorizations, then the corresponding extensions $H$ and $H'$ of $H_0+V$ in Theorem \rref{abstract theorem} coincide.
\end{proposition}

\begin{proof}
By the first resolvent identity for $H_0$, for $z_1,z_2\in\rho(H_0)$,
\begin{align*}
A'R_0(z_1)B-A'R_0(z_2)B&=(z_2-z_1)\,A'R_0(z_2)R_0(z_1)B.
\end{align*}
Since the right hand side has bounded (everywhere defined) closure by assumption~i), it follows that
$A'R_0(z_1)B$ has bounded closure if and only if $A'R_0(z_2)B$ does.
Denote
\[
Q(z):=\overline{AR_0(z)B},\quad Q'(z):=\overline{A'R_0(z)B'},\quad z\in\rho(H_0).
\]  
For $f\in\dom$, $g\in\dom'$, $z\in\rho(H_0)$, we then have the identities
\begin{align*}
F(z)Q(z)f&=A'R_0(z)B\,AR_0(z)Bf=A'R_0(z)B'\,A'R_0(z)Bf=Q'(z)F(z)f,\\
G(z)Q'(z)g&=AR_0(z)B'\,A'R_0(z)B'g=AR_0(z)B\,AR_0(z)B'g=Q(z)G(z)g,
\end{align*}
which extend to all $f\in\K$, $g\in\K'$ by continuity, due to ii). In particular, for all $z\in\rho(H)$,
\begin{align*}
F(z)(I_{\K}\pm Q(z))&=(I_{\K'}\pm Q'(z))F(z),\\
G(z)(I_{\K'}\pm Q'(z))&=(I_{\K}\pm Q(z))G(z).
\end{align*}
Using the identities above, one can check that if $-1\in\rho(Q(z))$, then $-1\in\rho(Q'(z))$ and vice versa, and 
\begin{align}
(I_{\K'}+Q'(z))^{-1}&=(I_{\K'}-Q'(z))+F(z)(I_{\K}+Q(z))^{-1}G(z),\label{inverseQprime}\\
(I_{\K}+Q(z))^{-1}&=(I_{\K}-Q(z))+G(z)(I_{\K'}+Q'(z))^{-1}F(z)\label{inverseQ}.
\end{align}
This proves that 
\[
\rho(H)\cap\rho(H_0)=\rho(H')\cap\rho(H_0)\neq \emptyset.
\]
Using formula \eqref{inverseQprime} and the equality $BA=B'A'$, we infer that on the linear manifold $\mathcal{C}\subset\H$, for all $z\in\rho(H_0)\cap\rho(H)$,
\begin{align*}
&R_0(z)B\left(I_{\K}+Q(z)\right)^{-1}AR_0(z)\\
&=R_0(z)VR_0(z)-R_0(z)VR_0(z)VR_0(z)\\
&\quad +R_0(z)VR_0(z)B'\left(I_{\K'}+Q'(z)\right)^{-1}A'R_0(z)VR_0(z)\\
&=R_0(z)B'(I_{\K'}-Q'(z)+Q'(z)\left(I_{\K'}+Q'(z)\right)^{-1}Q'(z))A'R_0(z)\\
&=R_0(z)B'\left(I_{\K'}+Q'(z)\right)^{-1}A'R_0(z).
\end{align*}
Since $\mathcal{C}$ is dense in $\H$, this identity extends to all of $\H$ by continuity if we replace $R_0(z)B$ and $R_0(z)B'$ by their (bounded) closures, and hence formula \eqref{resolvent factorized perturbation} for the resolvents of $H$ and $H'$ shows that 
\[
(H-z)^{-1}=(H'-z)^{-1},\quad z\in \rho(H)\cap\rho(H_0)=\rho(H')\cap\rho(H_0).\qedhere
\]
\end{proof}

\begin{proposition}\label{proposition construction for slowly decaying}
Let $H_0$ be the free Dirac operator \eqref{eq. Dirac op.} on $\H=L^2(\R)\otimes\C^2$, and let $V=(V_{ij})_{i,j=1}^2$ with $V_{ij}\in L^1(\R)+ L^{\infty}_0(\R)$ for $i,j=1,2$. 
For any decomposition 
\beq\label{eq. decomposition V}
V=W+X,\quad W_{ij}\in L^1(\R),\quad X_{ij}\in L^{\infty}_0(\R),
\eeq
define $A$, $B$ as in \eqref{eq. factorization 2} on their natural domain. 
Then all decompositions of the form \eqref{eq. decomposition V} give rise to compatible factorizations $V=BA$. Moreover, these factorizations are also compatible with the one in \eqref{eq. standard factorization}.
\end{proposition}

\begin{proof}
We only prove the first claim. The proof of the second one is analogous.
Let $W,W'\in (L^1(\R))^4$ and $X,X'\in (L^{\infty}_0(\R))^4$ be such that $$V=W+X=W'+X'.$$ It is easy to see  
that $A^{\sharp}R_0(z)$, $\overline{R_0(z)B^{\sharp}}$ and $\overline{A^{\sharp}R_0(z)B^{\sharp}}$ are all bounded; here, $A^{\sharp}$ stands for $A$ or $A'$, and $B^{\sharp}$ stands for $B$ or $B'$. This shows that the condition~i) of Definition \ref{def. compatible} is satisfied.

In order to check condition ii) of Definition \ref{def. compatible}, let $\Xi (\R)\subset L^2(\R)$ denote the linear submanifold of step functions $f:\R\to\C$. We set
\[
\mathcal{C}:=\Xi (\R)\otimes\C^2,\quad \mathcal{D}:=\Xi (\R)\otimes\C^4,\quad \mathcal{D}':=\Xi (\R)\otimes\C^4. 
\]
Clearly, $\mathcal{C}\subset\H$, $\mathcal{D}\subset\K$, $\mathcal{D}'\subset\K$ are dense. Here, we only show that 
\beq\label{eq. domain inclusion D}
\dom\subset\set{f\in\dom(B)}{R_0(z)Bf\in\dom(V)},\quad z\in\rho(H_0);
\eeq
the proofs of the other two inclusions in Definition \ref{def. compatible} ii) are similar.
Note that, since $X$ is bounded, we have
\[
\dom(B)=\dom(B_W)\oplus\H,\quad \dom(V)=\dom(W).
\]
Let $f:= \chi_{[a,b]}\otimes (\al,\be)^t$ for some $a<b$ and $\al,\be\in\C^2$. Then $f=f_1+f_2$ with $f_1=\chi_{[a,b]}\otimes (\al,0)^t$, $f_2=\chi_{[a,b]}\otimes (0,\be)^t$ and for any $\eps>0$
\[
\int_{\R} \|B(x)f_1(x)\|_{\C^2}^2 \rd x\leq |\al|^2\,\int_{a}^b \|V(x)\| \rd x\leq  |\al|^2\,(C_{\eps}+(b-a)\,\eps),
\]
whence $f\in\dom(B)$. Now let $z\in\rho(H_0)$ and set $g:=R_0(z)Bf$. Then
\begin{align*}
\|g(x)\|_{\C^2}\leq \eta\,|\al| \int_a^b \e^{-\im\,k(z)\,|x-y|}\,\|W(y)\|^{1/2}\rd y
+\eta\,|\be|\,\|X\| \int_a^b \e^{-\im\,k(z)\,|x-y|}\,\rd y
\end{align*}
where we abbreviated $\eta(|\Phi(z)|)$ by $\eta$.
For $h\in\dom(W^*)$, we have
\begin{align*}
|(W^*h,g)|\leq \int_{\R} \|W(x)\| \, \|h(x)\|_{\C^2} \, \|g(x)\|_{\C^2} \rd x
\leq\eta\,|\al|\, I_1(h)+\eta|\be|\,\|X\|\, I_2(h)
\end{align*}
where
\begin{align*}
I_1(h)&= \int_{\R}\int_a^b \|W(x)\| \, \|h(x)\|_{\C^2} \,\e^{-\im\,k(z)\,|x-y|}\,\|W(y)\|^{1/2}\rd y\rd x\\
&\leq \eta \,\|h\|\,\int_a^b \left(\int_{\R}\|W(x)\|^2 \,\e^{-2\,\im\,k(z)\,|x-y|} \rd x\right)^{1/2}\,\|W(y)\|^{1/2}\rd y\\
&\leq \eta\,\|h\|\, \left(\sup_{a\leq y\leq b}\int_{\R}\|W(x)\|^2 \,\e^{-2\,\im\,k(z)\,|x-y|} \rd x\right)^{1/2}\,\int_a^b \|W(y)\|^{1/2}\rd y\\
&\leq \eta\,\|h\|\,\left(\sup_{a\leq y\leq b}\int_{\R}\|W(x)\|^2 \,\e^{-2\,\im\,k(z)\,|x-y|} \rd x\right)^{1/2}\,(b-a)\,\int_a^b \|W(y)\|\rd y,
\end{align*}
and, similarly,
\begin{align*}
I_2(h)&= \int_{\R}\int_a^b \|W(x)\| \, \|h(x)\|_{\C^2} \,\e^{-\im\,k(z)\,|x-y|}\,\rd y\rd x\\
&\leq \eta\,\|h\|\,(b-a)\, \left(\sup_{a\leq y\leq b}\int_{\R}\|W(x)\|^2 \,\e^{-2\,\im\,k(z)\,|x-y|} \rd x\right)^{1/2}.
\end{align*}
The supremum in the above two estimates is finite; indeed,
repeated application of Young's inequality yields
\[
\sup_{a\leq y\leq b}\int_{\R}\|W(x)\|^2 \,\e^{-2\,\im\,k(z)\,|x-y|} \rd x\leq \|W\|_1^4\, \|\e^{-2\,\im\,k(z)\,|\cdot|}\|_{6/7}. 
\]
This shows that $g\in\dom(W^{**})=\dom(W)$.
The claim now follows from Proposition~\ref{corollary independend of factorization}.
\end{proof}

In order to prove the invariance of the essential spectrum under perturbations $V\in L^1(\R)+L_0^{\infty}(\R)$, we use that the norm of the $L_0^{\infty}(\R)$ part can be made arbitrarily small. By the same arguments as explained in the introduction, this implies that $\sigma(H)\setminus\sigma_{\rm e}(H)=\sigma_{\rm d}(H)$ also holds for unbounded $V$.
%in the theorems of Sections \ref{section Integrable potentials}--\ref{section Embedded eigenvalues and resonances}.

\begin{proposition}
Let $H_0$ be the free Dirac operator \eqref{eq. Dirac op.} on $\H=L^2(\R)\otimes\C^2$, and let $V=(V_{ij})_{i,j=1}^2$ with $V_{ij}\in L^1(\R)+ L^{\infty}_0(\R)$ for $i,j=1,2$.
Then 
\[
\sigma_{\rm e}(H)=\sigma_{\rm e}(H_0)=(-\infty,-m]\cup [m,\infty).
\]
\end{proposition}

\begin{proof}
Suppose first that $V_{ij}\in L^1(\R)$, and let $V=BA$ with $A$ and $B$ given by \eqref{eq. standard factorization}. By \eqref{eq. norm of AR0 and R0B}, $AR_0(z)$ and $\overline{R_0(z)B}$ are Hilbert-Schmidt operators, which implies that the resolvent difference $R(z)-R_0(z)$ is compact (even trace class). The equality of the essential spectra of $H_0$ and $H$ thus follows from \cite[Theorem IX.2.4]{EE}. 

If $V_{ij}\in L^1(\R)+ L_0^{\infty}(\R)$, we choose sequences $(W_n)_{n\in\N}\subset (L^1(\R))^4$ and $(X_n)_{n\in\N}\subset (L^{\infty}_0(\R))^4$ such that $V=W_n+X_n$ for all $n\in\N$ and $\|X_n\|\to 0$, $n\to\infty$. Furthermore, let 
\[
A_n:=\begin{pmatrix}A_{W_n}\\A_{X_n}\end{pmatrix},\quad B_n:=\begin{pmatrix}B_{W_n}\!\!&\!\! B_{X_n}\end{pmatrix},\quad Q_n(z):=\overline{A_n R_0(z)B_n},
\]
where e.g.\ $A_{W_n}:=|W_n|^{1/2}$, $B_{W_n}:=U_{W_n}|W_n|^{1/2}$, and $U_{W_n}$ is the partial isometry in the polar decomposition of $W_n$. By Proposition \ref{proposition construction for slowly decaying} it follows that
\[
R(z)=R_0(z)-\overline{R_0(z)B_n}\left(I_{\K}+Q_n(z)\right)^{-1}A_n R_0(z)
\]
is independent of $n$. Using the relation \eqref{inverseQprime} or \eqref{inverseQ}, we obtain
\[
R(z)-R_0(z)=S_n+T_n
\]
where each summand of $S_n$ contains at least one factor of $A_{W_n} R_0(z)$, $\overline{R_0(z)B_{W_n}}$ or $\overline{A_{W_n}R_0(z)B_{W_n}}$, and each summand of $T_n$ contains only factors of $A_{X_n}$, $B_{X_n}$ or~$R_0(z)$. This means that $S_n$ is compact (even Hilbert-Schmidt), while $\|T_n\|\to 0$ as $n\to\infty$. Therefore, $R(z)-R_0(z)$ is the norm limit of compact operators and hence compact itself. 
\end{proof}

%%%%%%%%%%%%%%%%%%%%%%%%%%%%%%%%%%%%

%\acknowlegdements
\noindent
{\bf Acknowledgements.} 
{\small The first author gratefully acknowledges the support of Schweizerischer Nationalfonds, SNF, through the postdoc stipend PBBEP2\_\_136596; the third author thanks for the support of SNF, grant no.\ 200021-119826/1, and of Deutsche Forschungsgemeinschaft, DFG, grant no.\ TR368/6-2. Both thank the Institut Mittag-Leffler for the support and kind hospitality within the RIP programme ({\it Research in Peace}), during which this manuscript was completed.}

\bibliographystyle{plain}
\bibliography{literatur2} % name of the bib file

\end{document}